\documentclass[hidelinks,12pt]{article}
\usepackage[a4paper, margin=2.5cm]{geometry} 

\usepackage{latexsym,amssymb,graphicx,amsthm,amsmath,mathtools,caption,calc,transparent,float,dsfont,textcomp,comment,bm,enumitem,realhats}
\usepackage[backend=biber,style=alphabetic,giveninits=true]{biblatex}
\usepackage[justification= centering]{caption}
\usepackage{hyperref}

\numberwithin{figure}{section}
\numberwithin{equation}{section}

\AtEveryBibitem{
 \clearlist{address}
 \clearfield{date}
 \clearfield{eprint}
 \clearfield{isbn}
 \clearfield{issn}
 \clearlist{location}
 \clearfield{month}
 \clearfield{series}
 \clearfield{doi}
 \clearfield{url}
 
 \ifentrytype{book}{}{
  \clearlist{publisher}
  \clearname{editor}
 }
}

\newtheorem{theorem}{Theorem}[section]
\newtheorem{proposition}[theorem]{Proposition}
\newtheorem{lemma}[theorem]{Lemma}
\newtheorem{corollary}[theorem]{Corollary}
\newtheorem{claim}{Claim}[theorem]
\theoremstyle{definition}
\newtheorem{definition}[theorem]{Definition}

\theoremstyle{remark}
\newtheorem{remark}[theorem]{Remark}

\newcommand{\half}{\frac{1}{2}}
\newcommand{\nats}{\mathbb N}

\newcommand{\reals}{\mathbb R}

\newcommand{\eps}{\varepsilon}
\newcommand{\ph}{\varphi}
\newcommand{\setin}{\subseteq}
\newcommand{\del}{\partial}

\makeatletter
\newcommand*\twiddle[1]{\mathpalette\wthelper{#1}}
\newcommand*\wthelper[2]{%
        \hbox{\dimen@\accentfontxheight#1%
                \accentfontxheight#11.3\dimen@
                $\m@th#1\widetilde{#2}$%
                \accentfontxheight#1\dimen@
        }%
}

\newcommand*\accentfontxheight[1]{%
        \fontdimen5\ifx#1\displaystyle
                \textfont
        \else\ifx#1\textstyle
                \textfont
        \else\ifx#1\scriptstyle
                \scriptfont
        \else
                \scriptscriptfont
        \fi\fi\fi3
}
\makeatother

\DeclareMathOperator{\lap}{\Delta}
\newcommand{\lie}{\mathcal{L}}

\newcommand{\heatop}{\left(\del_t-\lap_g\right)}

\newcommand{\mfd}{\mathcal{M}}
\newcommand{\metsp}{\mathcal{X}}

\DeclareMathOperator{\ric}{Ric}

\DeclareMathOperator{\riem}{Rm}
\DeclareMathOperator{\sectional}{Sec}
\DeclareMathOperator{\vol}{vol}
\DeclareMathOperator{\inj}{inj}

\title{On the uniqueness of Ricci flows from certain Alexandrov spaces}
\author{Laura Flower}
\date{2 October 2026}

\begin{document}

\maketitle

\begin{abstract}
    We consider the problem of uniqueness of Ricci flows, starting from singular initial data. Even when the initial data is smooth and smoothly obtained, in the complete noncompact case uniqueness is not known in full generality. When starting from nonsmooth initial data, very little is known, even in the compact case.
    
    We obtain a uniqueness result for Ricci flows with scaling-invariant curvature bounds, a suitable noncollapsing assumption, and sectional curvature lower bounds, coming out of (possibly singular) Reifenberg spaces.
    
    As a consequence of our result, and existence results obtained by Simon--Topping and Lai, there is a canonical way to define a smooth structure on any Reifenberg uniformly volume-noncollapsed PIC2 limit space. Further, we obtain a diffeomorphic version of Perelman's stability theorem for these limit spaces.
\end{abstract}

\section{Introduction}\label{SEC introduction}

We say that a smooth family of Riemannian metrics $(g(t))_{t\in(0,T]}$ on a manifold $\mfd$ is a \emph{Ricci flow} if
\[
\del_tg(t)=-2\ric(g(t))
\]
for $t\in(0,T]$. Often, we might expect to see a Ricci flow defined for $t\in[0,T]$, and in that case we consider the metric $g(0)$ to be the initial data of the flow. However, it is also of interest to consider a flow that attains some less regular choice of initial data in a weaker sense. Throughout most of this paper, we simply ask that the distance function $d_{g(t)}$ converges locally uniformly to some distance function $d_0$ on $\mfd$ as $t\searrow0$. However, in our case this is implied by a weaker notion of convergence which is more useful for applications (see \S\ref{SEC consequences}).

We are interested in the case where the initial data is a \emph{uniformly non-collapsed limit space}, that is, a pointed Gromov--Hausdorff limit\footnote{For the definition of pointed Gromov--Hausdorff convergence which we use, see \cite[Definition 8.1.1]{MetricGeom}} of complete manifolds with a uniform lower bound on the volume of all unit balls and some notion of a uniform curvature lower bound. The existence of Ricci flows coming out of such spaces has been established in many cases. For example, in \cite{SimonToppingMollification} Simon and Topping obtain existence of complete Ricci flows with a $C/t$ curvature bound coming out of complete uniformly non-collapsed 3D Ricci limit spaces. In \cite{Lai}, Lai demonstrated that the same proof also works for uniformly non-collapsed $PIC_1$-limit spaces in any dimension, with suitable modifications incorporating estimates of Bamler, Cabezas-Rivas and Wilking \cite{BCRW}. These flows can then be used to show that given any such limit space, a locally bi-H\"older map can be constructed from the space to a smooth manifold, and in particular, to the smooth manifold given at any positive time along the flow (see \cite{SimonToppingMollification}, \cite{Lai}, \cite{McLeodTopping}).

The next question, then, is whether such Ricci flows must be unique, and our main result provides a partial answer to this.

Our main result requires the following Reifenberg condition.
\begin{definition}
    A metric space $(\metsp,d_\metsp)$ is \emph{uniformly Reifenberg} if for every $\eps>0$, there exists $r_\eps>0$ such that for any $p\in \metsp$ and $0<r\leq r_\eps$, we have $d_{GH}(B_{r^{-1}d_0}(p,1),B_{\reals^n}(0,1))<\eps$, where $d_{GH}$ is the Gromov--Hausdorff distance, and $\reals^n$ is equipped with the Euclidean metric.
\end{definition}
\begin{remark}
    If $\metsp$ is a compact volume-non-collapsed Ricci limit space, then this condition is equivalent to requiring all tangent cones of $\metsp$ to be Euclidean, by Bishop--Gromov volume comparison and Cheeger--Colding theory (see \cite[Lemma 2.1]{DSS2}).
\end{remark}
This means that limit spaces with cone points are not Reifenberg, for example the ``ice cream cone" example given by replacing a small spherical cap of a sphere with an appropriate choice of cone (see \cite[\S3.3]{ToppingRLSSurvey}). However, there are Reifenberg limit spaces which are not Riemannian manifolds, for example a ``flying saucer", constructed by identifying two copies of a spherical cap along their boundaries, to form a ``seam". Such seams can also meet each other, provided they satisfy a suitable angle condition, as in the construction one might use to approximate a sphere by sewing together some flat panels of fabric. See Figure \ref{PIC Reifenberg examples} for pictures of these examples.

\begin{figure}[ht]
   \centering
   \def\svgscale{1}
\begingroup%
  \makeatletter%
  \providecommand\color[2][]{%
    \errmessage{(Inkscape) Color is used for the text in Inkscape, but the package 'color.sty' is not loaded}%
    \renewcommand\color[2][]{}%
  }%
  \providecommand\transparent[1]{%
    \errmessage{(Inkscape) Transparency is used (non-zero) for the text in Inkscape, but the package 'transparent.sty' is not loaded}%
    \renewcommand\transparent[1]{}%
  }%
  \providecommand\rotatebox[2]{#2}%
  \newcommand*\fsize{\dimexpr\f@size pt\relax}%
  \newcommand*\lineheight[1]{\fontsize{\fsize}{#1\fsize}\selectfont}%
  \ifx\svgwidth\undefined%
    \setlength{\unitlength}{405.52626242bp}%
    \ifx\svgscale\undefined%
      \relax%
    \else%
      \setlength{\unitlength}{\unitlength * \real{\svgscale}}%
    \fi%
  \else%
    \setlength{\unitlength}{\svgwidth}%
  \fi%
  \global\let\svgwidth\undefined%
  \global\let\svgscale\undefined%
  \makeatother%
  \begin{picture}(1,0.22740244)%
    \lineheight{1}%
    \setlength\tabcolsep{0pt}%
    \put(0,0){\includegraphics[width=\unitlength,page=1]{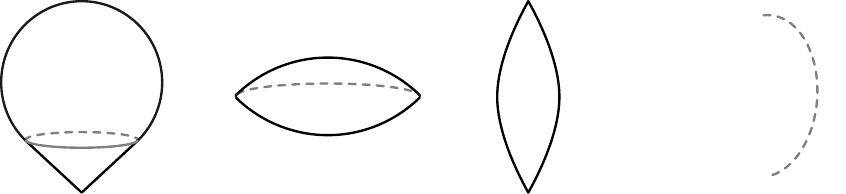}}%
    \put(0.64666424,0.0109429){\makebox(0,0)[lt]{\lineheight{1.25}\smash{\begin{tabular}[t]{l}$\times6$\end{tabular}}}}%
    \put(0,0){\includegraphics[width=\unitlength,page=2]{Reifenberg_examples.pdf}}%
    \put(0.62453249,0.17351199){\makebox(0,0)[t]{\lineheight{1.25}\smash{\begin{tabular}[t]{c}$\frac{\pi}{3}$\end{tabular}}}}%
    \put(0,0){\includegraphics[width=\unitlength,page=3]{Reifenberg_examples.pdf}}%
  \end{picture}%
\endgroup%

   \caption{A non-example and two examples of Reifenberg spaces: \\ an ``ice-cream cone", a ``flying saucer", and a ``six-panel sphere".}
   \label{PIC Reifenberg examples}
\end{figure}

We are now ready to state our main theorem.

\begin{theorem}\label{PROP Reifenberg uniqueness}
    Suppose $g(t)$, $\twiddle{g}(t)$, $t\in(0,T]$ are complete Ricci flows defined on smooth manifolds $\mfd$, $\twiddle{\mfd}$, with $d_{g(t)}\rightarrow d_0$, $d_{\twiddle{g}(t)}\rightarrow \twiddle{d}_0$ locally uniformly as $t\searrow0$ for some distance functions $d_0$, $\twiddle{d}_0$ on $\mfd$, $\twiddle{\mfd}$, which satisfy $d_0=\ph^*\twiddle{d}_0$ for some homeomorphism $\ph:\mfd\rightarrow\twiddle{\mfd}$. Suppose further that $g(t)$, $\twiddle{g}(t)$ satisfy 
    \begin{equation}\label{EQN assumptions inj}
        |\riem|\leq\frac{\alpha}{t}, \hspace{5ex} \inj\geq\sqrt{\frac{t}{\alpha}}, \hspace{5ex} \sectional\geq-1
    \end{equation}
    for some $\alpha>0$. If $(\mfd,d_0)$ is uniformly Reifenberg, then $\ph$ is a diffeomorphism and 
    \[
    g(t)\equiv\ph^*\twiddle{g}(t).
    \]
\end{theorem}

\begin{remark}
    For applications, we can replace most of the assumptions \eqref{EQN assumptions inj} with the assumption that the initial space is a uniformly volume non-collapsed limit space. We can also weaken the required convergence to the initial space to pointed Gromov--Hausdorff convergence; see \S\ref{SEC consequences}.
\end{remark}

\begin{remark}
    With the assumptions of Theorem \ref{PROP Reifenberg uniqueness}, the initial space $(\mfd,d_0)$ must be a uniformly non-collapsed limit space with a sectional curvature lower bound, and hence an Alexandrov space. However, in general we do not expect the existence of Ricci flows with sectional curvature lower bounds from all such spaces, because sectional curvature lower bounds are not preserved by the Ricci flow for dimensions $n>3$. Instead, in applications we must switch to the strictly stronger notion of a $PIC_2$ lower bound to ensure existence (see \S\ref{SEC consequences}).
\end{remark}

\subsection{Existing results}

In dimension two, the question of uniqueness is already settled. In the smooth case, 
Topping showed in \cite{Topping2D} that we get uniqueness under only the assumption that the flow is instantaneously complete. In the nonsmooth case, Topping and Yin obtained uniqueness of 2D Ricci flows starting from any non-atomic Radon measure in \cite{ToppingYin2D}.

However, in higher dimensions, much less is known.

In the compact case, uniqueness was demonstrated in the case of limit spaces with two-sided sectional curvature bounds by Corro, Moreno, and Zarei in \cite{UniquenessDoubleCurvatureBound}, provided the flows also have two-sided sectional curvature bounds.

Meanwhile, in the complete noncompact case, the question of uniqueness is nontrivial even when the initial data is smooth and smoothly obtained. In this setting, uniqueness of such flows with two-sided curvature bounds was obtained by Chen and Zhu in \cite{ChenZhu}, using a harmonic map heat flow strategy in which both flows are compared to a fixed reference metric. This strategy was streamlined by Kotschwar in \cite{Kotschwar} using energy methods, removing the need to consider the harmonic map heat flow at all. This further yielded uniqueness when both flows have a $C/t^\delta$ curvature bound for some $\delta\in(0,1/2)$, and with a volume growth condition at infinity.

In \cite{MCLeePaper}, Lee extended the strategy of Chen and Zhu to obtain uniqueness of such flows with a $C/t$ curvature bound, by using the harmonic map heat flow to directly compare the two flows. We will also make use of this strategy in order to establish the main result of this paper.

\subsection{Consequences of the main result}\label{SEC consequences}

In this section, we give some applications of Theorem \ref{PROP Reifenberg uniqueness}. For this, we will need to consider (uniformly non-collapsed) $PIC_2$ limit spaces. In particular, let $\mathcal{C}_{PIC_2}$ be the cone of weakly $PIC_2$ algebraic curvature operators (that is, algebraic curvature operators with non-negative complex sectional curvature; see \cite{ToppingPIC1Survey} for further exposition). We write $\riem\in\mathcal{C}_{PIC_2}$ if the curvature operator $\riem$ is weakly $PIC_2$. Taking $\mathrm{I}$ to be the identity curvature operator with scalar curvature $n(n-1)$, we say that $\riem$ satisfies a $PIC_2$ lower bound if there is $K\in\reals$ such that $\riem+K\mathrm{I}\in\mathcal{C}_{PIC_2}$. Note that the $PIC_2$ curvature condition implies both the $PIC_1$ curvature condition and nonnegative sectional curvature. Further, in dimension $3$, the $PIC_2$ curvature condition and nonnegative sectional curvature are equivalent, so the results in this section hold with a sectional curvature lower bound in this dimension.

\begin{definition}
    Let $r_0,v_0,K_0>0$. Then a \emph{($n$-dimensional) uniformly non-collapsed $PIC_2$ limit space} $(\metsp,d_\metsp)$ is a  pointed Gromov--Hausdorff limit of complete pointed Riemannian manifolds $(\mfd^n_i,g_i,p_i)$, satisfying
    \[
    \vol_{g_i}(B_{g_i}(y,r_0)) \geq v_0, \hspace{5ex} \riem(g_i)+K_0\mathrm{I}\in\mathcal{C}_{PIC_2}
    \]
    for all $i$ and at every $y\in\mfd_i$.

    We say that $(\metsp,d_\metsp)$ has a \emph{$PIC_2$ lower bound of $-K_0$}, and that its approximating manifolds have \emph{uniform volume lower bound $v_0$ at scale $r_0$}.
\end{definition}

We now give a consequence of Theorem \ref{PROP Reifenberg uniqueness} which will be more useful for applications.

\begin{theorem}\label{PROP main corollary}
    Let $(\metsp,d_\metsp)$ be a uniformly non-collapsed $PIC_2$ limit space which is uniformly Reifenberg.

    Now suppose that $(\mfd,g(t)),(\twiddle{\mfd},\twiddle{g}(t)),t\in(0,T]$ are complete Ricci flows. Assume further that there exists $p\in\mfd$, $\twiddle{p}\in\twiddle{\mfd}$, $x\in\metsp$, and a sequence of times $t_i\searrow0$, such that $(\mfd,g(t_i),p)$, $(\twiddle{\mfd},\twiddle{g}(t_i),\twiddle{p})$ pointed Gromov--Hausdorff converge to $(\metsp,d_\metsp,x)$ as $i\rightarrow\infty$. Suppose there is $K>0$ such that $g(t),\twiddle{g}(t)$ both satisfy
    \[
    \riem+K\mathrm{I}\in\mathcal{C}_{PIC_2}
    \]
    for $t\in(0,T]$.

    Then the distance functions $d_{g(t)},d_{\twiddle{g}(t)}$ both converge locally uniformly to some distance functions $d_0,\twiddle{d}_0$ as $t\searrow0$, such that the topologies induced by $d_0$, $\twiddle{d}_0$ agree with the manifold topologies of $\mfd$, $\twiddle{\mfd}$. Both $(\mfd,d_0)$ and $(\twiddle{\mfd},\twiddle{d}_0)$ are isometric to $(\metsp,d_\metsp)$ as metric spaces.
    
    Further, let $\ph:\mfd\rightarrow\twiddle{\mfd}$ be a homeomorphism such that $d_0=\ph^*\twiddle{d}_0$. Then the map $\ph:\mfd\rightarrow\twiddle{\mfd}$ is a diffeomorphism, and $g(t)\equiv\ph^*\twiddle{g}(t)$ for all $t$.
\end{theorem}

The proof of Theorem \ref{PROP main corollary}, given Theorem \ref{PROP Reifenberg uniqueness}, can be found in \S\ref{SEC proofs of consequences}.

It may seem surprising that we must require such a strong non-collapsing assumption as a uniform lower bound on the volume of \emph{every} unit ball. However, if we replace this assumption with a lower bound on the volume of just one unit ball, then the result fails. This can be seen by considering the example of a hyperbolic punctured torus given in \cite{ToppingCusp}. There are two different Ricci flows which come out of this space; the homothetically expanding one, but also a flow on a non-punctured torus, in which the cusp contracts. We note that in particular, the latter flow does not satisfy a $C/t$ curvature bound, so does not meet the conditions of Theorem \ref{PROP Reifenberg uniqueness}.

While Theorem \ref{PROP main corollary} is a statement about the uniqueness of Ricci flows, it is also possible to obtain the existence of such flows in a wider setting, as in the following theorem.

\begin{theorem}[{\cite[Theorem 1.8]{SimonToppingMollification}}, {\cite[Proof of Corollary 1.3]{Lai}}] \label{PROP RF existence}
    Let $r_0,v_0,K_0>0$, and let $(\metsp,d_\metsp,x)$ be a uniformly non-collapsed $PIC_2$ limit space with a $PIC_2$ lower bound of $-K_0$, such that the approximating manifolds have uniform volume lower bound $v_0$ at scale $r_0$.

    Then there exists a smooth manifold $\mfd$, a complete Ricci flow $g(t)$ on $\mfd$ for $t\in(0,T]$, where $T=T(r_0,v_0,K_0)$, and a continuous distance function $d_0$ on $\mfd$, such that $d_{g(t)}\rightarrow d_0$ locally uniformly as $t\searrow0$, and $(\mfd,d_0)$ is isometric to $(\metsp,d_\metsp)$ as metric spaces. Further, there exist some some $\alpha,v,K>0$, depending only on $r_0,v_0,K_0$, such that the Ricci flow satisfies
    \[
    |\riem(g(t))|\leq\frac{\alpha}{t}, \hspace{5ex} \vol_{g(t)}(B_{g(t)}(x,r_0)) \geq v, \hspace{5ex} \riem(g(t))+K\mathrm{I}\in\mathcal{C}_{PIC_2},
    \]
    and there are constants $\gamma_n,C_n>0$ such that
    \[
    d_{g(t_1)}-\gamma_n\sqrt{\alpha}\left(\sqrt{t_2}-\sqrt{t_1}\right) \leq d_{g(t_2)} \leq e^{C_nK(t_2-t_1)}d_{g(t_1)}
    \]
    for any $0\leq t_1\leq t_2\leq T$.
\end{theorem}

While we will not need it here, we note that the analogous result for a $PIC_1$ lower bound also holds (see \cite{Lai}).

By combining Theorem \ref{PROP main corollary} and Theorem \ref{PROP RF existence}, we can obtain some interesting corollaries. To make these easier to state, we give the following definition.

\begin{definition}
    Let $\mfd$ be a smooth manifold, and let $d_0$ be a distance function defined on $\mfd$, such that the topology generated by $d_0$ agrees with the manifold topology of $\mfd$. Assume further that $(\mfd,d_0)$ is a uniformly non-collapsed $PIC_2$ limit space. We say that $(\mfd,d_0)$ has a \emph{flow-compatible smooth structure} if the following is true. For some $T>0$, there is a complete Ricci flow $(g(t))_{t\in(0,T]}$ defined on $\mfd$, such that for some $K\geq0$ we have $\riem(g(t))+K\mathrm{I}\in\mathcal{C}_{PIC_2}$ for all $t\in(0,T]$, and the distance function $d_{g(t)}$ converges locally uniformly to $d_0$ as $t\searrow0$.
\end{definition}

\begin{remark}
    A priori, any uniformly non-collapsed $PIC_2$ limit space is merely a metric space. This definition is requiring such a space to additionally be endowed with a smooth structure, which can support a Ricci flow starting from $d_\metsp$.
\end{remark}

\begin{remark}
    By Theorem \ref{PROP RF existence}, every uniformly non-collapsed $PIC_2$ limit space can be endowed with a flow-compatible smooth structure.
\end{remark}

By applying theorem \ref{PROP main corollary}, we may obtain the following.

\begin{corollary}\label{PROP new Myers--Steenrod}
    Let $(\mfd,d_0), (\twiddle{\mfd},\twiddle{d}_0)$ be Reifenberg uniformly non-collapsed $PIC_2$ limit spaces with flow-compatible smooth structures. Then any isometry of metric spaces $\phi:(\mfd,d_0)\rightarrow(\twiddle{\mfd},\twiddle{d}_0)$ must be a diffeomorphism.
\end{corollary}

\begin{remark}
    Corollary \ref{PROP new Myers--Steenrod} an analogue of the Myers--Steenrod theorem for this setup in the Reifenberg case. This means that in this case, flow-compatible smooth structures are unique up to diffeomorphism.
\end{remark}

Using this, we may prove the following diffeomorphic version of Perelman's stability theorem (\cite{PerelmanStabilityTheorem}, see also \cite{StabilityTheoremAccessible}) in our setting.

\begin{theorem}\label{PROP diffeo stability thm}
    Let $(\mfd,d_0), (\twiddle{\mfd},\twiddle{d}_0)$ be compact Reifenberg $n$-dimensional uniformly non-collapsed $PIC_2$ limit spaces, with flow-compatible smooth structures and a $PIC_2$ lower bound of $-K$. Then there exists $\eps=\eps((\mfd,d_0),K)>0$ such that if $d_{GH}((\mfd,d_0), (\twiddle{\mfd},\twiddle{d}_0))<\eps$, then $\twiddle{\mfd}$ is diffeomorphic to $\mfd$.
\end{theorem}

For the proof of Theorem \ref{PROP diffeo stability thm}, see \S\ref{SEC proofs of consequences}.

We have the following corollary of Theorem \ref{PROP diffeo stability thm}, which gives an affirmative answer to \cite[Question 1.9]{WilkingSurvey} in our particular situation.

\begin{theorem}\label{PROP diffeo stability corollary}
    Let $(\mfd_k,g_k)_{k\in\nats}$ be a sequence of $n$-dimensional Riemannian manifolds with a uniform $PIC_2$ lower bound converging to a compact Reifenberg uniformly non-collapsed $PIC_2$ limit space. Then for sufficiently large $k_1$, $k_2$, the manifolds $\mfd_{k_1}$, $\mfd_{k_2}$ are diffeomorphic.
\end{theorem}

Analogous results to Theorem \ref{PROP diffeo stability thm} and Corollary \ref{PROP diffeo stability corollary} can also be obtained in the complete noncompact case, provided the conclusions are suitably localised.

\subsection{Proof strategy, and outline of this paper}

We now discuss the strategy we will use to prove Theorem \ref{PROP Reifenberg uniqueness}, the proof of which will occupy the rest of this paper. We wish to find a gauge in which the equations for $g(t)$ and $\twiddle{g}(t)$ become parabolic, by writing them as Ricci-DeTurck flows with respect to some choice of metric. As in \cite{MCLeePaper}, we compare them directly by writing $g(t)$ as a Ricci-DeTurck flow with respect to $\twiddle{g}(t)$ (since $\twiddle{g}(t)$ is already a Ricci-DeTurck flow with respect to itself). In particular, suppose we have a time-dependent family of diffeomorphisms $F:\mfd\times(0,T]\rightarrow\twiddle{\mfd}$ which satisfy the harmonic map heat flow equation
\begin{equation}\label{EQN HMHF}
    \del_t F_t = \lap_{g(t),\twiddle{g}(t)} F_t,
\end{equation}
where we write $F_t(\cdot):=F(\cdot,t)$.
Then the pushforward $\widehat{g}(t):=(F_t^{-1})^*g(t)$ is a \emph{Ricci-DeTurck flow} with respect to $\twiddle{g}(t)$, and satisfies the following equation.
\begin{equation}\label{EQN Ricci-DeTurck}
    \begin{cases}
        \del_t\widehat{g} = -2\ric(\widehat{g})+\lie_Vg;\\
        V^k = \widehat{g}^{ij}\left(\widehat{\Gamma}^k_{ij}-\twiddle{\Gamma}^k_{ij}\right).
    \end{cases}
\end{equation}
The hope is to obtain such a family $(F_t)_{t\in(0,T]}$ of diffeomorphisms, with good estimates, such that $F_t\rightarrow\ph$ as $t\searrow0$, where $\ph:\mfd\rightarrow\twiddle{\mfd}$ is the isometry of metric spaces in the assumptions of Theorem \ref{PROP Reifenberg uniqueness}. Then since \eqref{EQN Ricci-DeTurck} is uniformly parabolic, unlike the Ricci flow equation, we hope this will make it possible to obtain that $\widehat{g}(t)\equiv\twiddle{g}(t)$. This means that $F_t$ is an isometry for all $t\in(0,T]$, and is therefore constant in time by \eqref{EQN HMHF}, which gives the result.

This leaves the task of finding the family $F$ of diffeomorphisms. We are working with the disadvantage that $g(t),\twiddle{g}(t)$ are not smooth back to $t=0$, so we need to start our harmonic map heat flow at some positive time, and take limits as that time goes back to $0$. The na\"ive approach might be to consider the given homeomorphism $\ph:\mfd\rightarrow\twiddle{\mfd}$ at each of a sequence of times $t_i\searrow0$, take the solution to \eqref{EQN HMHF} starting from $\ph$ at each time $t_i$, and take the limit of these as $t_i\searrow0$. Unfortunately, while we will have control on how much the distance functions $d_{g(t)}$ and $\ph^*d_{\twiddle{g}(t)}$ can differ at each time $t$, thanks to distance distortion estimates (see \S\ref{SEC proofs of consequences}), we have no control on the extent to which $g(t)$ and $\ph^*\twiddle{g}(t)$ can differ. This means we have no control on the derivative of $\ph$ when seen as a map from $(\mfd,g(t))$ to $(\twiddle{\mfd},\twiddle{g}(t))$, which prevents us from taking limits of the desired flows, or obtaining useful estimates afterwards.

Therefore, before we can flow and take limits, we must ``smooth out" the map $\ph$. Once this is achieved, it is then possible to obtain the estimates needed to take limits and obtain the desired family $F$ of diffeomorphisms.

The structure of this paper is as follows. In \S\ref{SEC proofs of consequences}, we give the proofs of the various consequences of Theorem \ref{PROP Reifenberg uniqueness} mentioned in \S\ref{SEC introduction}. In \S\ref{SEC smoothing}, we obtain ``smoothed versions" of the map $\ph:\mfd\rightarrow\twiddle{\mfd}$, with derivative control, by making use of a result of Gromov, and parts of Katsuda's proof of this result which gives full details. In \S\ref{SEC HMHF}, we use results from \cite[\S2]{MCLeePaper} to establish that we can flow these maps by harmonic map heat flow. Further, we can use the parabolicity of \eqref{EQN HMHF} to obtain estimates for all derivatives of these flows, which will allow us to take limits. In \S\ref{SEC R-DT flow}, we change viewpoints, and use the parabolicity of \eqref{EQN Ricci-DeTurck} to show that the harmonic map heat flows remain close to being isometries of Riemannian manifolds. This uses ideas from \cite[\S3]{MCLeePaper}, and a result on local stability of the Ricci-DeTurck flow from \cite{CLL25}. Finally, in \S\ref{SEC uniqueness}, we make use of the Reifenberg condition to provide the final ingredient required to prove the main theorem.

\subsection{Acknowledgements}
The author wishes to thank her supervisors, Felix Schulze and Peter Topping, and Man Chun Lee, for many helpful and insightful discussions of this topic. The author is supported by the Warwick Mathematics Institute Centre for Doctoral Training, and gratefully acknowledges funding from the University of Warwick and the UK Engineering and Physical Sciences Research Council (Grant number: EP/W524645/1).

\subsection{AI usage}
All the mathematical ideas and strategies in this article were developed solely by the author. AI was only used in the very final stages of the preparation of this manuscript to identify possible improvements of the exposition. The author retains full responsibility for correctness of the final work.

\section{Proofs of consequences of the main theorem} \label{SEC proofs of consequences}

We now give the proofs of some of the consequences of Theorem \ref{PROP Reifenberg uniqueness} given in \S\ref{SEC consequences}.

We begin by proving Theorem \ref{PROP main corollary}, given Theorem \ref{PROP Reifenberg uniqueness}, which first means understanding how to use information about the initial data of a Ricci flow to establish properties of the flow.

\begin{lemma}\label{PROP non-collapsed flow from non-collapsed space, const radius}
    Let $(\mfd^n,g(t)), t\in(0,T]$ be a complete Ricci flow. Suppose that there is $p\in\mfd$ such that $(\mfd,d_{g(t_i)},p)$ pointed Gromov--Hausdorff converges to a pointed metric space $(\metsp,d_\metsp,x)$ for some sequence of times $t_i\searrow0$. Suppose that $g(t)$ satisfies
    \[
    \riem+K\mathrm{I}\in\mathcal{C}_{PIC_1}
    \]
    for all $t\in(0,T]$, for some $\alpha,K>0$. Assume further that $(\metsp,d_\metsp,x)$ is a volume non-collapsed Ricci limit space. That is, there exist $v_0,\twiddle{K}>0$ such that $(\metsp,d_\metsp,x)$ is a pointed Gromov--Hausdorff limit of complete
    pointed Riemannian manifolds $(\twiddle{\mfd}^n_i,g_i,p_i)$, satisfying
    \[
    \vol_{g_i}(B_{g_i}(y,1)) \geq v_0, \hspace{5ex} \ric(g_i)\geq-\twiddle{K}
    \]
    for all $i$ and at every $y\in\twiddle{\mfd}_i$.

    Then there exist some $\widehat{T}, \alpha, \eta_0>0$, all depending only on $v_0,K,n$,
    such that
    \[
    |\riem(g(t))|\leq\frac{\alpha}{t}, \hspace{5ex}  \vol_{g(t)}(B_{g(t)}(y,1)) \geq \eta_0, \hspace{5ex} \inj_{g(t)}\geq\sqrt{\frac{t}{\alpha}}
    \]
    at every $y\in\mfd$, and for all $t\in(0,\min\{T,\widehat{T}\}]$.
\end{lemma}

\begin{proof}
    Note that we have pointed Gromov--Hausdorff convergence of $(\mfd,g(t))$ and $(\twiddle{\mfd}_i,g_i)$ to $(\metsp,d_\metsp)$, with a Ricci lower bound in both cases. Then \cite[Theorem 5.9]{CheegerColding} allows us to pass the uniform volume lower bound from the $(\twiddle{\mfd}_i,g_i,p_i)$ to $(\metsp,d_\metsp,x)$, giving that $\mathcal{H}^n(B_{d_\metsp}(y,1))\geq v_0$ for every $y\in\metsp$. Next, we wish to pass the uniform volume lower bound from $(\metsp,d_\metsp,x)$ to $(\mfd,d_{g(t_i)},p)$.
    
    Fix $R>0$. We claim that there is $N_R\in\nats$ such that for each $i\geq N_R$, we have
    \[
    \vol_{g(t_i)}(B_{g(t_i)}(y,1)) \geq \half v_0
    \]
    for every $y\in B_{g(t_i)}(p,R)$. If not, then take a sequence of indices $i_j$ and points $y_j\in B_{g(t_{i_j})}(p,R)$ such that $\vol_{g(t_{i_j})}(B_{g(t_{i_j})}(y_j,1))<\half v_0$ for each $j$. Then, by passing to a subsequence such that under the Gromov--Hausdorff convergence $(\twiddle{\mfd}_i,g_i,p_i)\rightarrow(\metsp,d_\metsp,x)$, the points $y_j$ converge to some point $x_0$, we obtain a contradiction with \cite[Theorem 5.9]{CheegerColding}.
    
    Then given any $i\geq N_R$, we may apply \cite[Lemma 3.4]{Lai} starting from time $t_i$ to any point in $B_{g(t_i)}(p,R)$. This gives $\widehat{T}, \alpha, \eta_0>0$, all depending only on $v_0,K,n$,
    such that
    \[
    |\riem(g(t))|\leq\frac{\alpha}{t-t_i}, \hspace{5ex}  \vol_{g(t)}(B_{g(t)}(y,1)) \geq \eta_0, \hspace{5ex} \inj_{g(t)}\geq\sqrt{\frac{t-t_i}{\alpha}}
    \]
    at every $y\in B_{g(t_i)}(p,R)$,
    for all $t\in[t_i,\min\{T,t_i+\widehat{T}\})$. Using the curvature bound, by \cite[Corollary 3.3]{SimonToppingLocalControl} there is $\gamma_n>0$ depending only on $n$ such that
    \[
    B_{g(t_i)}(p,R) \supseteq B_{g(t)}\left(p,R-\gamma_n\sqrt{\alpha(t-t_i)}\right) \supseteq B_{g(t)}\left(p,R-\gamma_n\sqrt{\alpha\widehat{T}}\right).
    \]
    for all $t\in[t_i,\min\{T,\widehat{T}\})$, provided $R\geq\gamma_n\sqrt{\alpha\widehat{T}}$. Sending $i\rightarrow\infty$, then $R\rightarrow\infty$, gives the conclusion.
\end{proof}

Note that both the $PIC_1$ and nonnegative Ricci curvature conditions needed for Lemma \ref{PROP non-collapsed flow from non-collapsed space, const radius} are implied by the $PIC_2$ curvature condition (see \cite{ToppingPIC1Survey} for more information).
Therefore, if the conditions of Theorem \ref{PROP main corollary} hold, then both Ricci flows satisfy the scaling-invariant curvature bound $|\riem|\leq\frac{\alpha}{t}$, and the injectivity radius lower bound $\inj\geq\sqrt{\frac{t}{\alpha}}$. The scaling-invariant curvature bound, combined with the curvature lower bound on each flow which we are already assuming, will give us access to the following.

\begin{lemma}[Consequence of {\cite[Lemma 3.1]{SimonToppingMollification}}]\label{PROP distance distortion}
Let $(\mfd^n,g(t))_{t\in (0,T]}$, be a smooth complete Ricci flow, satisfying $\ric(\cdot,t) \geq -K , \quad |\riem(\cdot,t)| \leq \alpha/t$. Then
\begin{equation}\label{EQN dist distortion positive times}
    d_{g(t_1)}-\gamma_n\sqrt{\alpha}\left(\sqrt{t_2}-\sqrt{t_1}\right) \leq d_{g(t_2)} \leq e^{K(t_2-t_1)}d_{g(t_1)}
\end{equation}
for all $0<t_1\leq t_2\leq T$, and some $\gamma_n>0$ depending only on $n$. In particular, there is a distance function $d_0$ on $\mfd$ such that $d_t \to d_0$ locally uniformly as $t \searrow 0$, and
\[
d_0 - \gamma_n\sqrt{\alpha t} \leq d_{g(t)} \leq e^{Kt}d_0   \hspace{5ex} \text{ for all } t\in (0,T]  \text{ on } \mfd.
\]  

Furthermore, the topology of $\mfd$ induced by $d_0$ agrees with the manifold topology of $\mfd$. 
\end{lemma}

This now allows us to prove Theorem \ref{PROP main corollary}.

\begin{proof}[Proof of Theorem \ref{PROP main corollary}, given Theorem \ref{PROP Reifenberg uniqueness}]
    First, apply Lemma \ref{PROP non-collapsed flow from non-collapsed space, const radius}, to obtain
    \[
    |\riem(g(t))|\leq\frac{\alpha}{t}, \hspace{5ex} \inj_{g(t)}\geq\sqrt{\frac{t}{\alpha}}
    \]
    for some $\alpha>0$ and for $t\in(0,T]$, possibly after reducing $T$.

    Next, by applying Lemma \ref{PROP distance distortion}, we can upgrade from pointed Gromov--Hausdorff convergence to pointwise convergence of the distance function. In particular, we  may assume that the complete Ricci flows $(\mfd^n,g(t))$, $(\twiddle{\mfd}^n,\twiddle{g}(t))$, $t\in(0,T]$, are such that $d_{g(t)}\rightarrow d_0$, $d_{\twiddle{g}(t)}\rightarrow \twiddle{d}_0$ as $t\searrow0$, for distance functions $d_0$ on $\mfd$ and $\twiddle{d}_0$ on $\twiddle{\mfd}$ respectively. We also have that the topology induced by $d_0$ agrees with the manifold topology of $\mfd$, and that the topology induced by $\twiddle{d}_0$ agrees with the manifold topology of $\twiddle{\mfd}$. Then  $(\mfd,g(t),p)\rightarrow(\mfd,d_0,p)$, $(\twiddle{\mfd},\twiddle{g}(t),\twiddle{p})\rightarrow(\twiddle{\mfd},\twiddle{d}_0,\twiddle{p})$ as $t\searrow0$ also in a pointed Gromov--Hausdorff sense, so $(\mfd,d_0,p)$, $(\twiddle{\mfd},\twiddle{d}_0,\twiddle{p})$ are isometric to $(\metsp,d_\metsp,x)$ as metric spaces.

    Finally, take any homeomorphism $\ph:\mfd\rightarrow\twiddle{\mfd}$ such that $d_0=\ph^*\twiddle{d}_0$. Recall that the $PIC_2$ curvature condition is strictly stronger than nonnegative sectional curvature, so possibly after rescaling, all conditions of Theorem \ref{PROP Reifenberg uniqueness} are met. This gives the result.
\end{proof}

Finally, we give the proof of Theorem \ref{PROP diffeo stability thm}, using Theorem \ref{PROP main corollary}.

\begin{proof}[Proof of Theorem \ref{PROP diffeo stability thm}]
    Assume for contradiction that the result is not true. Then there must be a compact Reifenberg uniformly non-collapsed $PIC_2$ limit space $(\mfd,d_0)$, and a sequence $(\twiddle{\mfd}_i,\twiddle{d}_i)_{i\in\nats}$ of compact uniformly non-collapsed $PIC_2$ limit spaces, such that the following is true. All limit spaces have flow-compatible smooth structures and the same $PIC_2$ lower bound $-K$, and further, $d_{GH}((\mfd,d_0),(\twiddle{\mfd}_i,\twiddle{d}_i))\rightarrow0$, but none of the $\twiddle{\mfd}_i$ are diffeomorphic to $\mfd$.

    Next, let $R=\operatorname{diam}(\mfd,d_0)$. We claim that there is $v>0$ such that, after restricting to a tail of the sequence, each of the $(\twiddle{\mfd}_i,\twiddle{d}_i)$ have approximating manifolds with uniform volume lower bound $v$ at scale $2R$. If not, then by using the fact that the $(\twiddle{\mfd}_i,\twiddle{d}_i)$ are each Gromov--Hausdorff limits of $n$-dimensional smooth manifolds, we may find a sequence of $n$-dimensional smooth manifolds which Gromov--Hausdorff converge to $(\mfd,d_0)$, with volume converging to $0$. By \cite[Theorem 3.1]{CheegerColding}, this means that $\operatorname{dim}(\mfd,d_0)\leq n-1$, contradicting the fact that $\operatorname{dim}(\mfd,d_0)=n$ by \cite[Theorem 5.9]{CheegerColding} because $(\mfd,d_0)$ is noncollapsed. Hence we have the required uniform volume lower bound.

    For each $i$, take the Ricci flow $(\twiddle{\mfd}_i,\twiddle{g}_i(t))_{t\in(0,T]}$ coming out of $(\twiddle{\mfd}_i,\twiddle{d}_i)$ given by Theorem \ref{PROP RF existence}.  Here, corollary \ref{PROP new Myers--Steenrod} gives that the flow $\twiddle{g}_i(t)$ is defined on $\twiddle{\mfd}_i$, and not some other manifold. Note that $T=T(2R,v,K)>0$ is independent of $i$, and the flows all satisfy
    \begin{equation}\label{EQN diffeo stability pf curv bounds}
        |\riem(\twiddle{g}_i(t))|\leq\frac{\alpha_0}{t}, \hspace{5ex} \vol_{\twiddle{g}_i(t)}(B_{\twiddle{g}_i(t)}(x,2R)) \geq v_0, \hspace{5ex} \riem(\twiddle{g}_i(t))+K_0\mathrm{I}\in\mathcal{C}_{PIC_2}
    \end{equation}
    for all $t$, for some $\alpha_0,v_0,K_0$ which are also independent of $i$.

    We will argue as in the proof of \cite[Theorem 1.8]{SimonToppingMollification} that we may take a subsequential limit of these flows to get a Ricci flow coming out of $(\mfd,d_0)$.

    The uniform curvature bounds for positive times allow us to apply Hamilton's compactness theorem. That is, after choosing points $\twiddle{p}_i\in\twiddle{\mfd}_i$ and passing to a subsequence, we may find a smooth manifold $\twiddle{\mfd}$, a point $\twiddle{p}\in\twiddle{\mfd}$, and a complete Ricci flow $\twiddle{g}(t)$ on $\twiddle{\mfd}$, defined for $t\in(0,T]$, such that $(\twiddle{\mfd}_i,\twiddle{g}_i(t),\twiddle{p}_i)\rightarrow(\twiddle{\mfd},\twiddle{g}(t),\twiddle{p})$ in the smooth Cheeger-Gromov sense.
    We may pass the curvature bounds in \eqref{EQN diffeo stability pf curv bounds} to $\twiddle{g}(t)$.

    Then by Lemma \ref{PROP distance distortion}, there exists a distance function $\twiddle{d}_0$ on $\twiddle{\mfd}$ such that $d_{\twiddle{g}(t)}\rightarrow \twiddle{d}_0$ locally uniformly on $\twiddle{\mfd}$ as $t\searrow0$, and the topology on $\twiddle{\mfd}$ induced by $\twiddle{d}_0$ agrees with the manifold topology of $\twiddle{\mfd}$. Further, \eqref{EQN dist distortion positive times}, combined with the smooth convergence $(\twiddle{\mfd}_i,\twiddle{g}_i(t),\twiddle{p}_i)\rightarrow(\twiddle{\mfd},\twiddle{g}(t),\twiddle{p})$ for $t>0$, and the Gromov-Hausdorff convergence $(\twiddle{\mfd}_i,\twiddle{d}_i)\rightarrow(\mfd,d_0)$, gives that $(\twiddle{\mfd},\twiddle{d}_0)$ is isometric to $(\mfd,d_0)$ as metric spaces. Therefore, by Corollary \ref{PROP new Myers--Steenrod}, $\twiddle{\mfd}$ is diffeomorphic to $\mfd$. Hence we may instead write that $(\twiddle{\mfd}_i,\twiddle{g}_i(t),\twiddle{p}_i)\rightarrow(\mfd,g(t),p)$ in the smooth Cheeger-Gromov sense, where $g(t)$ is a Ricci flow coming out of $(\mfd,d_0)$, for some $p\in\mfd$.
    
    Since $\mfd$ is compact, the Cheeger-Gromov convergence gives that, after restricting to a tail of the sequence, we may find a sequence of diffeomorphisms $\psi_i:\mfd\rightarrow\twiddle{\mfd}_i$ such that $\psi_i^*\twiddle{g}_i(t)\rightarrow g(t)$ smoothly locally on $\mfd\times(0,T)$. In particular, $\mfd$ is diffeomorphic to the $\twiddle{\mfd}_i$, giving the required contradiction.
\end{proof}

\section{Smoothing maps}\label{SEC smoothing}

We now begin the proof of Theorem \ref{PROP Reifenberg uniqueness}.

Our first aim is to smooth out the map $\ph:(\mfd,g(t_i))\rightarrow(\twiddle{\mfd},\twiddle{g}(t_i))$ at fixed times $t_i$, using a result of Gromov. We begin with \cite[Theorem 1]{Katsuda}, which is an adaptation of \cite[8.19]{Gromov} to give full details. We note also that similar results are obtained using different methods in \cite{PetersAlternativeSmoothing} and \cite{GreeneWuAlternativeSmoothing}.

We first recall the following definition.

\begin{definition}
    A subset $B\setin\mfd$ is \emph{$\delta$-dense} if for every $p\in\mfd$, there exists $b\in B$ such that $d_\mfd(p,b)\leq\delta$.\\
    A subset $B\setin\mfd$ is \emph{$\delta$-discrete} if for every $b_1,b_2\in B$ with $b_1\neq b_2$, we have $d_\mfd(b_1,b_2)\geq\delta$.
\end{definition}

We will subsequently extend the following theorem to give a slightly stronger result, and this will require some reference to details of the proof, so we give a brief outline of the proof below.

\begin{theorem}[{\cite[Theorem 1]{Katsuda}}]\label{PROP Katsuda original}
    Given $n,\Delta,i_0>0$, $0<R<\half\min\{\Delta^{-\half},i_0\}$, for any $\delta>0$, there exists $\eps=\eps(n,\Delta,i_0,R;\delta)>0, a=a(n,\Delta,i_0,R;\delta)>0$ such that the following holds.

    Let $\mfd,\twiddle{\mfd}$ be complete $n$-dimensional Riemannian manifolds with $|\sectional|\leq\Delta$ and $\inj\geq i_0$. Suppose we have $\eps$-dense, $\frac{\eps}{10}$-discrete subsets $N[\eps]=\{m_i\}_{i=1}^{N_\eps}\setin \mfd$, $\twiddle{N}[\eps]=\{\twiddle{m}_i\}_{i=1}^{N_\eps}\setin \twiddle{\mfd}$ with the same (possibly infinite) number of members, such that
    \[
    |d_{\mfd}(m_i,m_j)-d_{\twiddle{\mfd}}(\twiddle{m}_i,\twiddle{m}_j)|<a
    \hspace{8ex} \text{whenever } 0<d_{\mfd}(m_i,m_j)\leq R.
    \]
    Then there exists a diffeomorphism $f:\mfd\rightarrow\twiddle{\mfd}$ such that
    \[
    ||df_x(\xi)|-1|<\delta \hspace{8ex} \text{for every } (x,\xi)\in T\mfd \text{ with } |\xi|=1.
    \]
\end{theorem}

\begin{proof}[Proof (sketch), see {\cite[Theorem 1]{Katsuda}}.]
     The idea is to embed both $\mfd$ and $\twiddle{\mfd}$ in a large-dimensional Euclidean space $\reals^{N_\eps}$, using a standardised choice of embedding. This embedding is essentially defined by taking localised distance coordinates on each manifold, so it is reasonable to expect that if the distance functions of $\mfd$, $\twiddle{\mfd}$ are similar, then their embeddings will be very close to each other in $\reals^{N_\eps}$, and a normal projection between them will be bi-Lipschitz.

    We now give more detail about this standardised choice of embedding, beginning with the compact case. Suppose $(\mfd,g)$ is such that $|\sectional(g)|\leq\Delta$ and $\operatorname{Inj(g)}\geq i_0$, and consider the $\eps$-dense, $\frac{\eps}{10}$-discrete subset $N[\eps]=\{m_i\}_{i=1}^{N_\eps}\setin \mfd$. Note that a first idea might be to define the embedding $\psi:\mfd\rightarrow\reals^{N_\eps}$ by
    \[(\psi(x))_j=d_\mfd^2(x,m_j),\]
    but we instead make a more sophisticated choice. We define $\psi:\mfd\rightarrow\reals^{N_\eps}$ by 
    \[(\psi(x))_j=h(d_\mfd(x,m_j)),\]
    where $h$ is a suitably chosen smooth cutoff function, which satisfies
    \begin{align*}
        h(t)\equiv1 \hspace{8ex}&\text{when }t\leq0,\\
        h(t)\equiv0 \hspace{8ex}&\text{when }t\geq r,\\
        -\frac{4}{r}\leq h'(t)<0 \hspace{8ex}&\text{for }0<t<r,
    \end{align*}
    for a particular choice of $r$. Since $h(t)\equiv0$ for large $t$, this now means that the study of $\psi$ is entirely a local matter.

    It can be shown that $\psi$ is a smooth embedding with controlled derivative.
    
    The map $\twiddle{\psi}:\twiddle{\mfd}\rightarrow\reals^{N_\eps}$ is defined analogously to $\psi$. Then $\psi(\mfd)$, $\twiddle{\psi}(\twiddle{\mfd})$ are sufficiently close in $\reals^{N_\eps}$ that the normal projection $P:\psi(\mfd)\rightarrow \twiddle{\psi}(\twiddle{\mfd})$ is well-defined, and can further be shown to be a smooth diffeomorphism. Then we take $f:=\twiddle{\psi}^{-1}\circ P\circ \psi$, and the required control on its derivative is obtained.

    The (complete) noncompact case can then be handled by fixing $m_0\in\mfd$, and taking maps $f_b:B_\mfd(m_0,b-2R)\rightarrow B_{\twiddle{\mfd}}(\twiddle{m}_0,b)$ for large $b>0$ in the same way as given above. By the locality of $\psi$, these maps for any two values of $b$ agree wherever both are defined, so we can take the limit as $b\rightarrow\infty$.
\end{proof}

Next, we use information about the proof of Theorem \ref{PROP Katsuda original} to obtain a slightly stronger result. In particular, if $\ph:\mfd\rightarrow\twiddle{\mfd}$ is a map for which we only have control on the distance distortion, then we can construct $f$ to also remain close to $\ph$.

\begin{theorem}\label{PROP Katsuda extended}
    Given $n,\Delta,i_0,\rho>0$, $0<R<\half\min\{\Delta^{-\half},i_0\}$, for any $\delta>0$, there exists $a=a(n,\Delta,i_0,\rho,R;\delta)>0$ such that the following holds.

    Let $\mfd,\twiddle{\mfd}$ be complete $n$-dimensional Riemannian manifolds with $|\sectional|\leq\Delta$ and $\inj\geq i_0$. Suppose we have a map $\ph:\mfd\rightarrow\twiddle{\mfd}$ (not necessarily continuous, injective or surjective) such that
    \[
    |d_{\mfd}(x,y)-d_{\twiddle{\mfd}}(\ph(x),\ph(y))|<a
    \hspace{8ex} \text{whenever } x,y\in\mfd, 0<d_{\mfd}(x,y)\leq R,
    \]
    \[
    d_{\twiddle{\mfd}}(\ph(x),\ph(y))\geq \rho
    \hspace{8ex} \text{whenever } x,y\in\mfd, d_{\mfd}(x,y)\geq R,
    \]
    \[
    \text{and }\twiddle{\mfd}=\bigcup_{x\in\mfd}B_{\twiddle{\mfd}}(\ph(x),a).
    \]
    Then there exists a diffeomorphism $f:\mfd\rightarrow\twiddle{\mfd}$ such that
    \[
    d_{\twiddle{\mfd}}(f(x),\ph(x))<\delta \hspace{8ex} \text{for every } x\in\mfd,
    \]
    \[
    ||df_x(\xi)|-1|<\delta \hspace{8ex} \text{for every } (x,\xi)\in T\mfd \text{ with } |\xi|=1.
    \]
\end{theorem}

\begin{proof}
    Suppose we have $\mfd,\twiddle{\mfd},\ph$ as in the statement of the theorem, which satisfy the required estimates for some $n,\Delta,i_0,a,\rho>0$, $0<R<\half\min\{\Delta^{-\half},i_0\}$. Let $\eps=\eps(n,\Delta,i_0,R;\delta)>0$ be as given by Theorem \ref{PROP Katsuda original}. In the proof of \cite[Theorem 1]{Katsuda}, $\eps$ may be reduced as much as needed, with the only impact being the need to further reduce $a$ as a result, and $a$ can be reduced as much as needed with no further consequences. Therefore, we may assume without loss of generality that $a\leq\frac{\eps}{10}\leq\rho$.
    
    Take a $\frac{2\eps}{10}$-discrete $\frac{8\eps}{10}$-dense subset $N[\eps]=\{m_i\}_{i=1}^{N_{\eps}}\setin \mfd$, where $N_\eps\in\nats\cup\{\infty\}$, and define $\twiddle{m}_i:=\ph(m_i)$ for each $i$. Then $\twiddle{N}[\eps]=\{\twiddle{m}_i\}_{i=1}^{N_{\eps}}\setin \twiddle{\mfd}$ and $N[\eps]$ are both $\frac{\eps}{10}$-discrete $\eps$-dense subsets of their respective manifolds. Further, note that
    \[
    |d_{\mfd}(m_i,m_j)-d_{\twiddle{\mfd}}(\twiddle{m}_i,\twiddle{m}_j)|<a
    \hspace{8ex} \text{whenever } 0<d_{\mfd}(m_i,m_j)\leq R.
    \]

    Hence, we may apply Theorem \ref{PROP Katsuda original} to see that provided $a$ is sufficiently small, depending on $n,\Delta,i_0,\rho,R,\delta$, there is a diffeomorphism $f:\mfd\rightarrow\twiddle{\mfd}$ satisfying the required derivative bound. It remains to establish that $f$ is distance-close to $\ph$, which requires details from the proof of \cite[Theorem 1]{Katsuda} (see the proof of Theorem \ref{PROP Katsuda original} above).
    
    The diffeomorphism $f$ is given by $f=\twiddle{\psi}^{-1}\circ P\circ \psi$, where $\psi$, $\twiddle{\psi}$ are embeddings of $\mfd$, $\twiddle{\mfd}$ respectively into $\reals^{N_{\eps}}$, and $P$ is the normal projection from $\psi(\mfd)$ to $\twiddle{\psi}(\twiddle{\mfd})$. The embedding $\psi:\mfd\rightarrow\reals^{N_\eps}$ is defined using $N[\eps]$, that is $\psi:\mfd\rightarrow\reals^{N_\eps}$ is given by
    \[(\psi(x))_j=h(d_\mfd(p,m_j)),\]
    where $h$ is a suitably chosen smooth cutoff function (see \cite[\S2]{Katsuda}). In particular, for $r=r(n,\Delta,i_0,R;\delta)\in(0,R)$, we have\footnote{The precise choice of $r$ is made in \cite[\S8]{Katsuda}.}
    \begin{align*}
        h(t)\equiv1 \hspace{8ex}&\text{when }t\leq0,\\
        h(t)\equiv0 \hspace{8ex}&\text{when }t\geq r,\\
        -\frac{4}{r}\leq h'<0 \hspace{8ex}&\text{for }t\in\reals.
    \end{align*}
    In \cite{Katsuda}, $r$ is chosen independently of $\eps,a$, while the choices of $\eps,a$ depend on $r$.
    We additionally stipulate that
    \[
    h\left(\frac{r}{2}\right)=\half.
    \]
    
    As in \cite[\S3]{Katsuda}, by volume comparison results (see, for example, \cite[Theorem 3.101]{GHLRiemannianGeometry}),
    there is $\widehat{N}_{\eps}=\widehat{N}_{\eps}(n,\Delta,i_0,R)$ such that $\sup_{x\in\mfd}\#(B_{2r}(x)\cap N[\eps]) \leq \widehat{N}_{\eps}$. Then the image under $\psi$ of an $r$-ball around any point $x\in\mfd$ is contained within a subspace of $\reals^{N_{\eps}}$ of dimension at most $\widehat{N}_{\eps}$.

    Given $x\in\mfd$, we have (after possibly reordering $N[\eps]$)
    \begin{align*}
        d_{\reals^{N_{\eps}}}^2(\psi(x),\twiddle{\psi}(\ph(x)))
        &=\sum_{i=1}^{N_{\eps}} \left( h(d_\mfd(x,m_i)) - h(d_{\twiddle{\mfd}}(\ph(x),\twiddle{m}_i)) \right)^2 \\
        &\leq\sum_{i=1}^{\widehat{N}_{\eps}} \Big( a\sup_t|h'(t)| \Big)^2 \\
        &\leq \left( \frac{4a}{r}\right)^2 \widehat{N}_{\eps},
    \end{align*}
    so
    \begin{equation}\label{EQN Katsuda distance x to phi}
        d_{\reals^{N_{\eps}}}(\psi(x),\twiddle{\psi}(\ph(x)))\leq\frac{4a}{r}\widehat{N}_{\eps}^\half
    \end{equation}
    for each $x\in\mfd$, where $d_{\reals^{N_{\eps}}}$ is the Euclidean distance in $\reals^{N_\eps}$.

    In particular, we have that $d_{\reals^{N_\eps}}(\psi(\mfd),\twiddle{\psi}(\twiddle{\mfd}))\leq\frac{4a}{r}\widehat{N}_{\eps}^\half$. Recalling that $P$ is the normal projection and $P(\psi(x))=\twiddle{\psi}(f(x))$, this means that
    \begin{equation}\label{EQN Katsuda distance x to f}
        d_{\reals^{N_{\eps}}}(\psi(x),\twiddle{\psi}(f(x)))\leq\frac{4a}{r}\widehat{N}_{\eps}^\half
    \end{equation}
    for each $x\in\mfd$.
    
    Combining \eqref{EQN Katsuda distance x to phi} and \eqref{EQN Katsuda distance x to f} gives
    \begin{equation}\label{EQN Katsuda distance f to phi}
        d_{\reals^{N_{\eps}}}(\twiddle{\psi}(f(x)),\twiddle{\psi}(\ph(x)))\leq\frac{8a}{r}\widehat{N}_{\eps}^\half.
    \end{equation}

    We now use this to establish a bound for $d_{\twiddle{\mfd}}(f(x),\ph(x))$, following the method of \cite[Lemma 4.2]{Katsuda}. Assume for contradiction that
    \begin{equation}\label{EQN Katsuda contradiction assumption}
        d_{\twiddle{\mfd}}(f(x),\ph(x))\geq2r.
    \end{equation}
    As in \cite[Lemma 3.1]{Katsuda}, by volume comparison results (again see \cite[Theorem 3.101]{GHLRiemannianGeometry}) we may find $C_1=C_1(n,\Delta,i_0,R,r)>0$ such that $\inf_{x\in\twiddle{\mfd}}\#(B_{r/2}(x)\cap N[\eps]) \geq C_1\widehat{N}_{\eps}$. Although we do not need it here, we note that $C_1$ can be chosen independent of $\eps$. Then if $m_i\in B_{r/2}(\ph(x))\cap N[\eps]$, then $h(d_{\twiddle{\mfd}}(\ph(x),m_i))\geq\half$ by our choice of $h$, and $h(d_{\twiddle{\mfd}}(f(x),m_i))=0$ by \eqref{EQN Katsuda contradiction assumption}. Therefore
    \begin{align*}
        d_{\reals^{N_\eps}}^2(\twiddle{\psi}(f(x)),\twiddle{\psi}(\ph(x)))
        &=\sum_{i=1}^{N_{\eps}} \left( h(d_\mfd(f(x),m_i)) - h(d_{\twiddle{\mfd}}(\ph(x),\twiddle{m}_i)) \right)^2 \\
        &\geq \frac{1}{4}C_1\widehat{N}_{\eps},
    \end{align*}
    so
    \[
    d_{\reals^{N_\eps}}(\twiddle{\psi}(f(x)),\twiddle{\psi}(\ph(x))) \geq \half C_1^\half\widehat{N}_{\eps}^\half.
    \]
    This contradicts \eqref{EQN Katsuda distance x to f} if we reduce $a$. Hence \eqref{EQN Katsuda contradiction assumption} cannot hold, so
    \[
    d_{\twiddle{\mfd}}(f(x),\ph(x))\leq2r.
    \]
    Since the choice of $r$ already depends on $\delta$, we may further require $r\leq\delta/2$, giving the result.
\end{proof}

\section{Ricci-harmonic map heat flow} \label{SEC HMHF}

The aim of this section is to obtain an existence result, which will ultimately allow us to flow the ``smoothed-out" maps obtained in \S\ref{SEC smoothing} by the Ricci-harmonic map heat flow \eqref{EQN HMHF}. In order to achieve this, we will make use of the precise statements of several results from \cite{MCLeePaper}, and a variation of a localised maximum principle from \cite{LeeTamMaxPrinc}, which are collected in the appendix for convenience.

We note that since we will ultimately be applying the results in this section starting from a small positive time, the required global bounds on the curvature and its derivatives will hold for our application. Further, while the existence time depends on these bounds, for our application we will be iterating this result to get the longest possible existence time with good derivative estimates, so this is not a concern.

Finally, recall that given a time-dependent family of diffeomorphisms $F:\Omega\times[0,T]\rightarrow\twiddle{\mfd}$, where $\Omega\setin\mfd$, and a particular $t\in[0,T]$, we define $F_t$ by $F_t(\cdot):=F(\cdot,t)$. In particular, $F_0(\cdot):=F(\cdot,0)$. Further, quantities such as $|dF_t|$, $|D^kdF_t|$ are always evaluated with respect to the relevant metrics at time $t$ unless otherwise stated, that is $|dF_t|=|dF_t|_{g(t),\twiddle{g}(t)}$.

\begin{proposition}\label{PROP Global HMHF existence}
    Suppose $(\mfd^n, g(t))$, $(\twiddle{\mfd}^n, \twiddle{g}(t))$, $t \in [0, T]$ are two smooth complete solutions to the Ricci flow, such that for all $t\in[0,T]$ and some $K,\alpha,\rho>0$,
    \begin{enumerate}[label=(\roman*)]
        \item $\ric(\twiddle{g}(t))\geq-K\twiddle{g}(t)$,
        \item $\sup_{\mfd}|\riem(g(t))| + \sup_{\twiddle{\mfd}}|\riem(\twiddle{g}(t))| \leq \min\{\alpha t^{-1}, \rho^{-2}\}$,
        \item $\sup_{\mfd}|\nabla\riem(g(t))| \leq \rho^{-3}$,
        \item $\sup_{\twiddle{\mfd}}|\twiddle{\nabla}^k\riem(\twiddle{g}(t))| < \infty$ for all $k\geq0$.  \label{LIST Global HMHF existence bdd curv derivs assumption}
    \end{enumerate}
    Let $\Lambda_0>0$, and let $f:\mfd\rightarrow\twiddle{\mfd}$ be a smooth map with $|df|_{g(0),\twiddle{g}(0)}^2\leq\Lambda_0$.
    
    Then there exists $T_0(n,\Lambda_0)>0$, and a smooth solution $F:\mfd\times[0,\min\{T,T_0\rho^2\}]\rightarrow\twiddle{\mfd}$ to \eqref{EQN HMHF} with $F_0=f$, such that
    \begin{equation}\label{EQN Global HMHF existence, deriv result}
        |dF_t|^2\leq10\Lambda_0
    \end{equation}
    for all $t\in[0,\min\{T,T_0\rho^2\}]$. Further, for each $k\geq1$, there exists $L_k(n,\alpha,\Lambda_0,k)>0$ such that
    \begin{equation}\label{EQN Global HMHF existence, higher deriv result}
        |D^kdF_t|^2 \leq L_kt^{-k}
    \end{equation}
    on $\mfd\times(0,\min\{T,T_0\rho^2\}]$, and
    \begin{equation}\label{EQN Global HMHF existence, dist result}
        d_{\twiddle{g}(t)}(f(x),F_{t}(x)) \leq 2e^{2Kt}\sqrt{L_1t}
    \end{equation}
    for all $(x,t)\in\mfd\times[0,\min\{T,T_0\rho^2\}]$.
\end{proposition}

\begin{proof}
    Let $x_0\in\mfd$ and $R>0$. Apply Proposition \ref{PROP Local HMHF existence} to obtain $T_0(n,\Lambda_0)>0$ and a solution $F^R:B_{g(0)}(x_0,8R)\times[0,\min\{T,T_0\rho^2\}]\rightarrow\twiddle{\mfd}$ to \eqref{EQN HMHF}, such that $F^R_0=f$ and
    \[
    |dF^R_t|^2\leq 10\Lambda_0
    \]
    for all $t\in[0,\min\{T,T_0\rho^2\}]$. We will ultimately send $R\rightarrow\infty$ and take a subsequential limit of the $F^R$ to obtain the required $F$. We begin by obtaining the estimates which will establish \eqref{EQN Global HMHF existence, higher deriv result} and \eqref{EQN Global HMHF existence, dist result} after taking limits.
    
    Provided $R$ is large enough, by Proposition \ref{PROP Lee HMHF bound on derivs} we have
    \begin{equation}\label{EQN HMHF global existence, higher deriv bound}
        |D^kdF^R_t|^2 \leq L_kt^{-k}
    \end{equation}
    on $B_{g(0)}(x_0,2R)\times(0,\min\{T,T_0\rho^2\}]$ for each $k$, where $L_k=L_k(n,\alpha,\Lambda_0,k)$. This will give \eqref{EQN Global HMHF existence, higher deriv result}.

    Now take $x\in B_{g(0)}(x_0,2R)$, and consider the curve $\gamma(s)=F^R_{s}(x), s\in[0,\min\{T,T_0\rho^2\}]$. Then using the estimate above and the Ricci flow equation, for $t\in(0,\min\{T,T_0\rho^2\}]$ we have
    \begin{align}
        d_{\twiddle{g}(t)}(f(x),F^R_t(x))
        &\leq \int_0^t|\gamma'(s)|_{\twiddle{g}(t)} \,\mathrm{d}s \nonumber\\
        &= \int_0^t|\del_sF^R_s(x)|_{\twiddle{g}(t)} \,\mathrm{d}s \nonumber\\
        &\leq \int_0^te^{2K(t-s)}|\del_sF^R_s(x)|_{\twiddle{g}(s)} \,\mathrm{d}s \nonumber\\
        &\leq e^{2Kt} \int_0^t\sqrt{L_1}s^{-1/2} \,\mathrm{d}s \nonumber\\
        &= 2e^{2Kt}\sqrt{L_1t}. \label{EQN HMHF dist calc}
    \end{align}
    This will give \eqref{EQN Global HMHF existence, dist result}.

    Finally, we must obtain the estimates required to give smooth convergence when we send $R\rightarrow\infty$. Currently, \eqref{EQN HMHF global existence, higher deriv bound} gives smooth convergence of the $F^R$ away from $t=0$, but to get smooth convergence back to $t=0$, we need local bounds on the derivatives of $F^R$ which are uniform in $t$.

    Let $k\geq1$ and $r>0$. Suppose $R$ is large enough that $R>2^kr$. We will inductively obtain a bound for $|D^kdF^R_t|^2$ on $B_{g(0)}(x_0,r)\times[0,\min\{T,T_0\rho^2\}]$. Suppose we have already obtained the bounds
    \[
    |D^\ell dF^R_t|^2 \leq C_{n,r,g,\twiddle{g},f,T,\ell}
    \]
    on $B_{g(0)}(x_0,2^{k-\ell}r)\times[0,\min\{T,T_0\rho^2\}]$ for $0\leq\ell<k$. We are assuming that $g$ is smooth back to $t=0$, so for each $\ell$ we have $|\nabla^\ell\riem(g(t))|\leq C_{r,g,\ell}$ on $B_{g(0)}(x_0,2r)\times[0,\min\{T,T_0\rho^2\}]$. Assumption (\ref{LIST Global HMHF existence bdd curv derivs assumption}) of the proposition gives that $|\twiddle{\nabla}^\ell\riem(\twiddle{g}(t))|\leq C_{\twiddle{g},\ell}$ on $\twiddle{\mfd}\times[0,\min\{T,T_0\rho^2\}]$, for each $\ell$. Hence by Lemma \ref{PROP Lee HMHF evo eqn for derivs}, we have
    \[
    \heatop |D^kdF^R_t|^2 \leq C_{n,r,g,\twiddle{g},f,T,k}\left(|D^kdF^R_t|^2+1\right)
    \]
    on $B_{g(0)}(x_0,2r)\times[0,\min\{T,T_0\rho^2\}]$. Further, since $f$ is smooth, we have
    \[
    |D^kdf|^2 \leq C_{r,f,k}
    \]
    on $B_{g(0)}(x_0,2r)$.
    
    Now take $\eta:[0,\infty)\rightarrow\reals$ to solve
    \[
    \eta'(t) = C_{n,r,g,\twiddle{g},f,T,k}\left(\eta(t)+1\right)
    \]
    with $\eta(0)=C_{r,f,k}$. Then $|D^kdF^R_t|^2-\eta$ solves
    \[
    \heatop \left(|D^kdF^R_t|^2-\eta\right) \leq C_{n,r,g,\twiddle{g},f,T,k}\left(|D^kdF^R_t|^2-\eta\right)
    \]
    on $B_{g(0)}(x_0,2r)\times[0,\min\{T,T_0\rho^2\}]$, with $|D^kdF^R_t|^2-\eta\leq0$ when $t=0$. Finally, by \eqref{EQN HMHF global existence, higher deriv bound}, we have the bound
    \[
    |D^kdF^R_t|^2-\eta \leq L_kt^{-k}
    \]
    on $B_{g(0)}(x_0,2r)\times[0,\min\{T,T_0\rho^2\}]$.
    Therefore, by Theorem \ref{PROP Lee Tam max princ improved}, we can choose $\ell$, depending only on $n,r,g,\twiddle{g},f,T,k$, such that
    \begin{align}
        |D^kdF^R_t|^2
        &\leq \eta(t) + r^{-2(\ell+k)}t^\ell \nonumber\\
        &\leq C_{n,r,g,\twiddle{g},f,T,k} \label{EQN HMHF global existence, local uniform deriv bd}
    \end{align}
    on $B_{g(0)}(x_0,r)\times[0,\min\{T,T_0\rho^2\}]$. This concludes the induction.

    Hence for fixed $r>0$ and $k\in\nats$, if $R$ is sufficiently large then we have a bound for $|D^kdF^R_t|^2$ on $B_{g(0)}(x_0,r)\times[0,\min\{T,T_0\rho^2\}]$ which is uniform in $R$. Therefore we may send $R\rightarrow\infty$ and take a smooth subsequential limit of the $F^R$ to obtain the desired $F$.
\end{proof}

In \S\ref{SEC R-DT flow}, we will be comparing a solution to \eqref{EQN Ricci-DeTurck} on $\mfd$ obtained using $F_t$, to a solution obtained in another way. This means we need to know that $F_t$ continues to exist for a long time, which motivates the following.

\begin{proposition}\label{PROP Long global HMHF existence}
    Suppose $(\mfd^n, g(t))$, $(\twiddle{\mfd}^n, \twiddle{g}(t))$, $t \in [0, T]$ are two smooth complete solutions to the Ricci flow, such that for all $t\in[0,T]$ and some $K,\alpha>0$,
    \begin{enumerate}[label=(\roman*)]
        \item $\ric(\twiddle{g}(t))\geq-K\twiddle{g}(t)$,
        \item $\sup_{\mfd}|\riem(g(t))| + \sup_{\twiddle{\mfd}}|\riem(\twiddle{g}(t))| \leq \alpha t^{-1}$,
        \item $\sup_{\mfd}|\riem(g(t))| + \sup_{\mfd}|\nabla\riem(g(t))| < \infty$,
        \item $\sup_{\twiddle{\mfd}}|\twiddle{\nabla}^k\riem(\twiddle{g}(t))| < \infty$ for all $k\geq0$.
    \end{enumerate}
    Let $\lambda_0>0$, and let $f:\mfd\rightarrow\twiddle{\mfd}$ be a smooth map with $|df|_{g(0),\twiddle{g}(0)}^2\leq\lambda_0$. Let $\Lambda_0>10\lambda_0$.
    
    Then there exists a smooth solution $F:\mfd\times[0,\widehat{T}]\rightarrow\twiddle{\mfd}$ to \eqref{EQN HMHF} with $F_0=f$, where $0<\widehat{T}\leq T$ is the maximal time such that
    \begin{equation}\label{EQN Long global HMHF existence, deriv result}
        |dF_t|^2\leq\Lambda_0
    \end{equation}
    for $t\in[0,\widehat{T}]$.
    
    Further, for each $k\geq1$, there exists $L_k(n,\alpha,\Lambda_0,k)>0$ such that
    \begin{equation}\label{EQN Long global HMHF existence, higher deriv result}
        |D^kdF_t|^2 \leq L_kt^{-k}
    \end{equation}
    on $\mfd\times(0,\widehat{T}]$, and
    \begin{equation}\label{EQN Long global HMHF existence, dist result}
        d_{\twiddle{g}(t)}(f(x),F_{t}(x)) \leq 2e^{2Kt}\sqrt{L_1t}
    \end{equation}
    for all $(x,t)\in\mfd\times[0,\widehat{T}]$.
\end{proposition}

\begin{proof}
    Take $\ell$ such that $10^{\ell-1}\lambda_0\leq\Lambda_0\leq10^{\ell}\lambda_0$. Since $\sup_{\mfd}|\riem(g(t))| + \sup_{\mfd}|\nabla\riem(g(t))| < \infty$ and $\sup_{\twiddle{\mfd}}|\riem(\twiddle{g}(t))| < \infty$, we can find $\rho>0$ such that the hypotheses of Proposition \ref{PROP Global HMHF existence} apply.
    
    Iteratively apply the proposition to obtain a solution for $t\in[0,\min\{T,T_\ell\rho^2\}]$, with $|dF_t|^2\leq10^\ell\lambda_0$ for all $t$, for some $T_\ell>0$. Then either there exists a time $t\geq T_0\rho^2>0$ when $|dF_t|^2>\Lambda_0$, or $|dF_t|^2\leq\Lambda_0$ for all $t\in[0,\min\{T,T_\ell\rho^2\}]$. In the latter case, continue to iteratively apply the proposition, making use of the initial bound $|dF_t|^2\leq\Lambda_0$, until either we have a solution with the required derivative bound for all $t\in[0,T]$, or until there is a time when $|dF_t|^2>\Lambda_0$. Now (in any case), take $\widehat{T}\in(0,T]$ to be the largest time such that $|dF_t|^2\leq\Lambda_0$ for $t\in[0,\widehat{T}]$.
    
    Finally, note that \eqref{EQN Long global HMHF existence, higher deriv result} follows by applying Proposition \ref{PROP Lee HMHF bound on derivs}, and \eqref{EQN Long global HMHF existence, dist result} can be obtained by repeating the calculation \eqref{EQN HMHF dist calc}.
\end{proof}

\section{Ricci-DeTurck flow}\label{SEC R-DT flow}

In \S\ref{SEC HMHF}, we used the parabolicity of \eqref{EQN HMHF} to obtain the existence of, and estimates for, a harmonic map heat flow $F$. In this section, we change perspectives to instead consider the metric $(F_t^{-1})^*g(t)$, which is a Ricci-DeTurck flow with respect to $\twiddle{g}(t)$. The parabolicity of \eqref{EQN Ricci-DeTurck} allows us to obtain further estimates for $(F_t^{-1})^*g(t)$ and, consequently, $F$.

First, we must ensure that the harmonic map heat flow obtained in Proposition \ref{PROP Long global HMHF existence} continues to be a diffeomorphism across the whole manifold for some small positive time. For this, we require the following evolution equation.

\begin{lemma}[Lee, {\cite[Lemma 2.5, first part]{MCLeePaper}}]\label{PROP Lee RDT flow evo eqn}
    There exists $10^{-2}>\overline{\eps}_0(n)>0$ such that the following holds: Suppose $(\twiddle{\mfd}^n,\twiddle{g}(t))_{t\in[0,T]}$ is a smooth solution to the Ricci flow, and $(\widehat{g}(t))_{t\in[0,T]}$ is a smooth solution to \eqref{EQN Ricci-DeTurck} with respect to $\twiddle{g}(t)$, such that
    \[
    (1-\overline{\eps}_0)\twiddle{g}(t) \leq \widehat{g}(t) \leq (1+\overline{\eps}_0)\twiddle{g}(t)
    \]
    for all $t$. Then $h:=\widehat{g}-\twiddle{g}$ satisfies
    \[
    \left(\del_t-\widehat{g}^{ij}\twiddle{\nabla}_i\twiddle{\nabla}_j\right) |h|^2 \leq -|\twiddle{\nabla}h|^2 + C_n|\riem(\twiddle{g})|\,|h|^2.
    \]
\end{lemma}

This allows us to prove the following.

\begin{proposition}\label{PROP HMHF short time diffeo}
    Suppose $(\mfd^n, g(t))$, $(\twiddle{\mfd}^n, \twiddle{g}(t))$, $t \in [0, T]$ are two smooth complete solutions to the Ricci flow, such that for all $t\in[0,T]$ and some $A,\alpha>0$,
    \[
    |\riem(g(t))|\leq\frac{\alpha}{t}, \hspace{5ex } |\riem(\twiddle{g}(t))|\leq\frac{\alpha}{t}, \hspace{5ex } |\riem(\twiddle{g}(t))|\leq A.
    \]
    Let $F:\mfd\times[0,T]\rightarrow\twiddle{\mfd}$ be a solution to \eqref{EQN HMHF}, and write $f:=F_0$. Suppose we have $\Lambda_0,\Lambda_1>0$ such that $|dF_t|^2\leq\Lambda_0$ for all $t\in[0,T]$, and $d_{\twiddle{g}(t)}(f(x),F_t(x))\leq\Lambda_1\sqrt{t}$ for all $(x,t)\in\mfd\times[0,T]$.

    Then there is $\overline{\eps}_1=\overline{\eps}_1(n)>0$, $c=c(n)>0$, such that the following is true. Suppose $f$ is a diffeomorphism which satisfies
    \[
    \bigg(1-\frac{1}{8\sqrt{n}}\overline{\eps}_1\bigg)g(0) \leq f^*\twiddle{g}(0) \leq \bigg(1+\frac{1}{8\sqrt{n}}\overline{\eps}_1\bigg)g(0).
    \]
    Then for all $t\in[0,T]\cap[0,cA^{-1}]$, we have that $F_t$ is a diffeomorphism and
    \[
    (1-\overline{\eps}_1)g(t) \leq F_t^*\twiddle{g}(t) \leq (1+\overline{\eps}_1)g(t).
    \]
\end{proposition}

\begin{proof}
    Before we begin, note that by Proposition \ref{PROP Lee HMHF bound on derivs}, there is $L_3=L_3(n,\alpha,\Lambda_0)>0$ such that for each $0\leq m\leq3$, we have
    \[
    |D^mdF_t|^2\leq L_3t^{-m}
    \]
    for all $t\in[0,T]$.

    To establish the proposition, we follow a large proportion of the induction argument in \cite[\S3]{MCLeePaper}, with a few small modifications. Since the argument involves defining many constants which must not depend on each other in a circular way, we repeat some of the details for the reader's convenience. We begin by defining those constants.
    \begin{enumerate}[label=(\roman*)]
        \item $\overline{\eps}_1:=\frac{1}{3}\overline{\eps}_0(n)$, where $\overline{\eps}_0$ is from Lemma \ref{PROP Lee RDT flow evo eqn}. This is the analogue of $\eps_3$ in \cite[\S3]{MCLeePaper}.
        \item $\tau_0:=\min\left\{T,A^{-1}C_n^{-1}\log2\right\}$, where $C_n$ is the constant from Lemma \ref{PROP Lee RDT flow evo eqn},
        \item $\widehat{L}:= 2\sqrt{\Lambda_0L_3} + 2\alpha(n-1)(\Lambda_0+1)$,
        \item $\gamma:=\gamma(n)$ from Lemma \ref{PROP distance distortion},
        \item $\twiddle{\alpha}:=\max\{\alpha,2n,C_nAT\}$, where again $C_n$ is the constant from Lemma \ref{PROP Lee RDT flow evo eqn},
        \item $\widehat{T}_1:=\widehat{T}(n,2\overline{\eps}_1,\twiddle{\alpha},\twiddle{\alpha}+1,0)$, where $\widehat{T}$ is from Theorem \ref{PROP Lee Tam max princ improved twisted},
        \item $\mu:=\min\{\overline{\eps}_1\widehat{L}^{-1},1\}$,
        \item $\Lambda:=\max\left\{\sqrt{\Lambda_0}\left(\mu\sqrt{\frac{L_3}{1-3\overline{\eps}_1}(1+\mu)^\alpha}+\gamma\sqrt{\alpha}\right),\ e^{(n-1)A\tau_0}\Lambda_1\sqrt{1+\mu},\ \sqrt{\frac{1+\mu}{\widehat{T}_1}},\ 4\sqrt{1+\mu}\,\overline{\eps}_1^{-1}\right\}$.
    \end{enumerate}
    
    Now let $p\in\twiddle{\mfd}$, and take any $r_0>0$. Then by continuity, there exists some $t_0\in(0,1)$, depending on $r_0$, such that 
    \[(1-\overline{\eps}_1)g(t) \leq F_t^*\twiddle{g}(t) \leq (1+\overline{\eps}_1)g(t)
    \]
    in the preimage $f^{-1}\left(B_{\twiddle{g}(0)}(p,r_0)\right)$ for all $0<t<t_0$. Then for each $k\in\nats$, we define
    \[
    t_k:=(1+\mu)t_{k-1}=(1+\mu)^kt_0,
    \]
    \[
    r_k:=r_{k-1}-4\Lambda\sqrt{t_{k-1}}=r_0-4\Lambda\cdot\frac{\sqrt{t_k}-\sqrt{t_0}}{\sqrt{1+\mu}-1}.
    \]
    We will prove the following statement by induction.
    \begin{quote}
        For each $k$, provided $r_k>0$, we have that
        \[
        (1-\overline{\eps}_1)g(t) \leq F_t^*\twiddle{g}(t) \leq (1+\overline{\eps}_1)g(t)
        \]
        on $f^{-1}\left(B_{\twiddle{g}(0)}(p,r_k)\right)$ for all $t\in[0,t_k]\cap[0,\tau_0]$. Further, $F_t$ is a diffeomorphism onto its image when restricted to $f^{-1}\left(B_{\twiddle{g}(0)}(p,r_k)\right)$ for all $t\in[0,t_k]\cap[0,\tau_0]$.
    \end{quote}

    Note that the base case $k=0$ follows from the choice of $t_0$. To begin the inductive step, suppose the statement holds for some $k$, and assume that $r_{k+1}>0,t_k\leq\tau_0$. We establish the inductive step through three claims.

    First, we obtain some control on $F_t^*\twiddle{g}$ in terms of $g$. This is the first step towards establishing that $F_t$ is a diffeomorphism onto its image, which will allow us to consider $(F_t^{-1})^*g(t)$. It will also allow us to use the evolution equation in Lemma \ref{PROP Lee RDT flow evo eqn}.

    \begin{claim}[see {\cite[Claim 3.2]{MCLeePaper}}]\label{PROP HMHF short time diffeo, local diffeo claim}
        We have that
        \[
        (1-2\overline{\eps}_1)g(t) \leq F_t^*\twiddle{g}(t) \leq (1+2\overline{\eps}_1)g(t)
        \]
        on $f^{-1}\left(B_{\twiddle{g}(0)}(p,r_k)\right)$ for all $t\in[0,t_{k+1}]\cap[0,\tau_0]$.
    \end{claim}
    \begin{proof}[Proof of claim]
        This is identical to the proof of \cite[Claim 3.2]{MCLeePaper}, but with changes to some of the constants, so we give the details for convenience.
    
        From the inductive assumption, we have that
        \[
        (1-\overline{\eps}_1)g(t_k) \leq F_{t_k}^*\twiddle{g}(t_k) \leq (1+\overline{\eps}_1)g(t_k)
        \]
        in $f^{-1}\left(B_{\twiddle{g}(0)}(p,r_k)\right)$.
        
        For $t\in[t_k,t_{k+1}]\cap[0,\tau_0]$, let $\ph(t):=1-\overline{\eps}_1-\widehat{L}t_k^{-1}(t-t_k)$, and note that $1-2\eps_1\leq\ph\leq1-\eps_1$ for $t\in[t_k,t_{k+1}]\cap[0,\tau_0]$ by our choice of $\mu$.
        Let $\twiddle{h}(t):=F_t^*\twiddle{g}(t)-\ph(t)g(t)$. Then for any fixed vector $v$, we have $\twiddle{h}(t_k)(v,v)\geq0$ by the inductive assumption. For $t\in[t_k,t_{k+1}]\cap[0,\tau_0]$ we have
        \begin{align*}
            \del_t\left(\twiddle{h}(t)(v,v)\right)
            &= \del_t \Big( \twiddle{g}(dF_t(v),dF_t(v)) - \ph\,g(v,v) \Big)\\
            &= 2\twiddle{g}(dF_t(v),(\del_tdF_t)(v)) - 2F_t^*\twiddle{\ric}(v,v) + 2\ph\ric(v,v) - \ph'g(v,v)  \\
            &\geq |v|_g^2 \left( -\left(2\sqrt{\Lambda_0L_3} +2\alpha(n-1)(\Lambda_0+1) \right)t^{-1} -\ph' \right) \\
            &\geq 0.
        \end{align*}
        Hence $\twiddle{h}(t)(v,v)\geq0$, and therefore $F_t^*\twiddle{g}(t)\geq(1-2\overline{\eps}_1)g(t)$, for all $t\in[t_k,t_{k+1}]\cap[0,\tau_0]$. The upper bound is established in the same way.
    \end{proof}

    This allows us to establish that on a smaller ball, $F_t$ is a diffeomorphism onto its image.

    \begin{claim}[see {\cite[Claim 3.3]{MCLeePaper}}]\label{PROP HMHF short time diffeo, diffeo claim}
        $F_t$ is a diffeomorphism onto its image when restricted to $f^{-1}\left(B_{\twiddle{g}(0)}(p,r_k-\Lambda\sqrt{t_k})\right)$ for all $t\in[0,t_{k+1}]\cap[0,\tau_0]$.
    \end{claim}
    \begin{proof}[Proof of claim]
        Repeating the proof of \cite[Claim 3.3]{MCLeePaper}, with the same constants, gives that $F_t$ is a diffeomorphism onto its image when restricted to
        \[
        U:=\left\{ x\in\mfd : B_{g(0)}(x,\twiddle{\Lambda}\sqrt{t_k}) \setin f^{-1}(B_{\twiddle{g}(0)}(p,r_k)) \right\}
        \]
        for all $t\in[0,t_{k+1}]\cap[0,\tau_0]$, where
        \[
        \twiddle{\Lambda}:=\mu\sqrt{\frac{L_3}{1-3\overline{\eps}_1}(1+\mu)^\alpha}+\gamma\sqrt{\alpha}.
        \]

        Since $|df|^2\leq\Lambda_0$, we have that $B_{g(0)}(x,\twiddle{\Lambda}\sqrt{t_k}) \setin f^{-1}(B_{\twiddle{g}(0)}(f(x),\twiddle{\Lambda}\sqrt{\Lambda_0t_k}))$ for all $x\in\mfd$. Therefore
        \[
        f^{-1}\left(B_{\twiddle{g}(0)}\left(p,r_k-\twiddle{\Lambda}\sqrt{\Lambda_0t_k}\right)\right) \setin U.
        \]
        Our choice of $\Lambda\geq\twiddle{\Lambda}\sqrt{\Lambda_0}$ establishes Claim \ref{PROP HMHF short time diffeo, diffeo claim}.
    \end{proof}

    Hence we may now consider $(F_t^{-1})^*g(t)$, and make use of the fact that it solves the Ricci-DeTurck flow equation with respect to $\twiddle{g}(t)$.

    \begin{claim}[see {\cite[Claim 3.4]{MCLeePaper}}]\label{PROP HMHF short time diffeo, better local diffeo claim}
        We have that
        \[
        (1-\overline{\eps}_1)g(t) \leq F_t^*\twiddle{g}(t) \leq (1+\overline{\eps}_1)g(t)
        \]
        on $f^{-1}\left(B_{\twiddle{g}(0)}(p,r_k-4\Lambda\sqrt{t_k})\right)$ for all $t\in[0,t_{k+1}]\cap[0,\tau_0]$.
    \end{claim}
    \begin{proof}
        For any $x\in\mfd,t\in[0,t_{k+1}]\cap[0,\tau_0]$, we have
        \begin{align}
            d_{\twiddle{g}(0)}(f(x),F_t(x))
            &\leq e^{(n-1)A\tau_0} \, d_{\twiddle{g}(t)}(f(x),F_t(x)) \nonumber\\
            &\leq e^{(n-1)A\tau_0} \Lambda_1\sqrt{t} \nonumber\\
            &\leq e^{(n-1)A\tau_0} \Lambda_1\sqrt{(1+\mu)t_k} \nonumber\\
            &\leq \half\Lambda\sqrt{t_k} \label{EQN HMHF short time diffeo, max dist moved}
        \end{align}
        by the Ricci flow equation and our choice of $\Lambda$.

        We claim that
        \begin{equation}\label{EQN HMHF short time diffeo, image containment}
            B_{\twiddle{g}(0)}(p,r_k-2\Lambda\sqrt{t_k}) \Subset F_t\left(f^{-1}\left(B_{\twiddle{g}(0)}(p,r_k-\Lambda\sqrt{t_k})\right)\right)
        \end{equation}
        for all $t\in[0,t_{k+1}]\cap[0,\tau_0]$. Suppose that $x\in\del\left(F_t\left(f^{-1}\left(B_{\twiddle{g}(0)}(p,r_k-\Lambda\sqrt{t_k})\right)\right)\right)$. Then $x=F_t(y)$ for some $y\in\del\left(f^{-1}\left(B_{\twiddle{g}(0)}(p,r_k-\Lambda\sqrt{t_k})\right)\right)$ by claim \ref{PROP HMHF short time diffeo, local diffeo claim}, so $d_{\twiddle{g}(0)}(p,f(y))=r_k-\Lambda\sqrt{t_k}$ because $f$ is a global diffeomorphism. Then $d_{\twiddle{g}(0)}(p,x)\geq r_k-2\Lambda\sqrt{t_k}$ by \eqref{EQN HMHF short time diffeo, max dist moved}. Hence
        \[
        \del\left(F_t\left(f^{-1}\left(B_{\twiddle{g}(0)}(p,r_k-\Lambda\sqrt{t_k})\right)\right)\right) \cap B_{\twiddle{g}(0)}(p,r_k-2\Lambda\sqrt{t_k}) = \emptyset
        \]
        for $t\in[0,t_{k+1}]\cap[0,\tau_0]$. Since the containment \eqref{EQN HMHF short time diffeo, image containment} holds when $t=0$, by continuity it must therefore hold for all $t\in[0,t_{k+1}]\cap[0,\tau_0]$.

        Therefore by Claim \ref{PROP HMHF short time diffeo, diffeo claim} and \eqref{EQN HMHF short time diffeo, image containment}, we may consider $\widehat{g}(t):=(F_t^{-1})^*g(t)$ on $B_{\twiddle{g}(0)}(p,r_k-2\Lambda\sqrt{t_k})$, for $t\in[0,t_{k+1}]\cap[0,\tau_0]$. This satisfies \eqref{EQN Ricci-DeTurck} with respect to $\twiddle{g}(t)$, and applying Claim \ref{PROP HMHF short time diffeo, local diffeo claim}, we further have that
        \[
        (1-3\overline{\eps}_1)\twiddle{g}(t) \leq (1+2\overline{\eps}_1)^{-1}\twiddle{g}(t) \leq \widehat{g}(t) \leq (1-2\overline{\eps}_1)^{-1}\twiddle{g}(t) \leq (1+3\overline{\eps}_1)\twiddle{g}(t),
        \]
        because $\overline{\eps}_1$ is sufficiently small.
        
        Let $h:=\widehat{g}-\twiddle{g}$. Then we may apply Claim \ref{PROP HMHF short time diffeo, local diffeo claim} and Lemma \ref{PROP Lee RDT flow evo eqn} to see that
        \[
        \left(\del_t-\widehat{g}^{ij}\twiddle{\nabla}_i\twiddle{\nabla}_j\right) |h|^2 \leq C_nA\,|h|^2
        \]
        on $B_{\twiddle{g}(0)}(p,r_k-2\Lambda\sqrt{t_k})$, for $t\in[0,t_{k+1}]\cap[0,\tau_0]$. Recall that when $t=0$ we have
        \[
        \bigg(1-\frac{1}{4\sqrt{n}}\overline{\eps}_1\bigg)\twiddle{g}(0)
        \leq \widehat{g}(0)
        \leq \bigg(1+\frac{1}{4\sqrt{n}}\overline{\eps}_1\bigg)\twiddle{g}(0)
        \]
        because $\overline{\eps}_1$ is sufficiently small. Therefore $|h|^2\leq\overline{\eps}_1^2$/16 when $t=0$.

        Next, we apply a local maximum principle to control how quickly $|h|^2$ can grow, thus obtaining Claim \ref{PROP HMHF short time diffeo, better local diffeo claim}. We already have $|h|^2\leq n(1+3\overline{\eps}_1)^2\leq2n$ for $t\in[0,t_{k+1}]\cap[0,\tau_0]$, so we may apply Theorem \ref{PROP Lee Tam max princ improved twisted}. We apply the theorem with $r=\Lambda\sqrt{t_k}$ and $\ell=\twiddle{\alpha}+1$ to $\ph:=|h|^2-\frac{1}{16}\overline{\eps}_1^2\exp(C_nAt)$, to see that
        \[
        |h|^2-\frac{1}{16}\overline{\eps}_1^2\exp(C_nAt) \leq \Lambda^{-2\ell}t_k^{-\ell}t^\ell
        \]
        on $B_{\twiddle{g}(0)}(p,r_k-3\Lambda\sqrt{t_k})$, for $t\in[0,t_{k+1}]\cap[0,\widehat{T}_1\Lambda^2t_k]\cap[0,\tau_0]=[0,t_{k+1}]\cap[0,\tau_0]$ by our choice of $\Lambda$. Hence, since $\ell\geq1$ and by our choices of $\tau_0\leq A^{-1}C_n^{-1}\log2$, $\Lambda\geq4\sqrt{1+\mu}\,\overline{\eps}_1^{-1}$, we have 
        \[
        |h|^2 \leq \frac{\overline{\eps}_1^2}{4}
        \]
        on the same set. Therefore
        \[
        (1-\overline{\eps}_1)\widehat{g}(t) \leq (1+\overline{\eps}_1/2)^{-1}\widehat{g}(t) \leq \twiddle{g}(t) \leq (1-\overline{\eps}_1/2)^{-1}\widehat{g}(t) \leq (1+\overline{\eps}_1)\widehat{g}(t)
        \]
        on $B_{\twiddle{g}(0)}(p,r_k-3\Lambda\sqrt{t_k})$, for $t\in[0,t_{k+1}]\cap[0,\tau_0]$ because $\overline{\eps}_1$ is sufficiently small.
        
        Finally, using the containment of preimages
        \[
        F_t^{-1}\left(B_{\twiddle{g}(0)}(p,r_k-3\Lambda\sqrt{t_k})\right)
        \supseteq f^{-1}\left(B_{\twiddle{g}(0)}(p,r_k-4\Lambda\sqrt{t_k})\right)
        \]
        proves Claim \ref{PROP HMHF short time diffeo, better local diffeo claim}.
    \end{proof}

    This concludes the induction argument. Finally, we fix the finite $k\in\nats$ such that $t_{k-1}\leq\tau_0\leq t_k$. Then $t_k\leq(1+\mu)\tau_0$, so
    \[
    r_k\geq r_0 - 4\Lambda\cdot\frac{\sqrt{1+\mu}\sqrt{\tau_0}}{\sqrt{1+\mu}-1}\geq r_0-\frac{16\Lambda\sqrt{\tau_0}}{\mu}.
    \]
    Therefore
    \[
    (1-\overline{\eps}_1)g(t) \leq F_t^*\twiddle{g}(t) \leq (1+\overline{\eps}_1)g(t)
    \]
    on $f^{-1}\left(B_{\twiddle{g}(0)}(p,r_0-16\Lambda\mu^{-1}\sqrt{\tau_0})\right)$ for all $t\in[0,\tau_0]$, and  $F_t$ is a diffeomorphism onto its image when restricted to $f^{-1}\left(B_{\twiddle{g}(0)}(p,r_0-16\Lambda\mu^{-1}\sqrt{\tau_0})\right)$ for all $t\in[0,\tau_0]$. Since the quantity $16\Lambda\mu^{-1}\sqrt{\tau_0}$ is independent of $r_0$ and $t_0$, we may send $r_0\rightarrow\infty$ to conclude the proof of Proposition \ref{PROP HMHF short time diffeo}.
\end{proof}

We now know that, at least for a short time, the solutions we have obtained to \eqref{EQN HMHF} remain diffeomorphisms which are close to isometries. This means we can now begin to consider the Ricci-DeTurck flow $(F_t^{-1})^*g(t)$. We will use the following result concerning stability of the Ricci-DeTurck flow.

\begin{theorem}[Chan--Lai--Lee, {\cite[Theorem 1.3]{CLL25}}]\label{PROP CLL flow existence}
    Let $\twiddle{\mfd}$ be a complete manifold. Suppose $\twiddle{g}(t)$, $t\in[0,T]$ is a smooth Ricci flow on $\twiddle{\mfd}$ such that \eqref{EQN assumptions inj} holds for some $\alpha>0$. Then there exists $\overline{\eps}_2(n,\alpha), \Lambda(n,\alpha)>0$ so that the following holds: if $\widehat{g}_0$ is a smooth metric so that
    \begin{equation}\label{EQN RDT stability closeness condition}
        (1-\eps)\twiddle{g}(0)\leq\widehat{g}_0\leq(1+\eps)\twiddle{g}(0) \hspace{3ex} \text{on }\twiddle{\mfd}
    \end{equation}
    for some $0\leq\eps<\overline{\eps}_2$, then there is a solution $\widehat{g}$ to \eqref{EQN Ricci-DeTurck} on $\twiddle{\mfd}\times[0,\min\{T,1\}]$ such that $\widehat{g}(0)=\widehat{g}_0$ and
    \[
    (1-\Lambda\eps)\twiddle{g}(t)\leq\widehat{g}(t)\leq(1+\Lambda\eps)\twiddle{g}(t) \hspace{3ex} \text{on }\twiddle{\mfd}
    \]
    for all $t\in[0,\min\{T,1\}]$.
    
    Further, if $\widehat{g}_1(t), \widehat{g}_2(t)$ are both solutions to \eqref{EQN Ricci-DeTurck} such that
    \[
    (1-\eps)\twiddle{g}(t)\leq\widehat{g}_1(t), \widehat{g}_2(t)\leq(1+\eps)\twiddle{g}(t) \hspace{3ex} \text{on }\twiddle{\mfd}
    \]
    for all $t\in[0,T]$,
    and there is $\twiddle{\eps}\geq0$ such that
    \[
    (1-\twiddle{\eps})\widehat{g}_1(0)\leq\widehat{g}_2(0)\leq(1+\twiddle{\eps})\widehat{g}_1(0) \hspace{3ex} \text{on }\twiddle{\mfd},
    \]
    then
    \[
    (1-\Lambda\twiddle{\eps})\widehat{g}_1(t)\leq\widehat{g}_2(t)\leq(1+\Lambda\twiddle{\eps})\widehat{g}_1(t) \hspace{3ex} \text{on }\twiddle{\mfd}
    \]
    for all $t\in[0,\min\{T,1\}]$.
    In particular, this gives uniqueness of solutions to \eqref{EQN Ricci-DeTurck} which satisfy \eqref{EQN RDT stability closeness condition}.
\end{theorem}

This allows us to prove the key ingredient for the main result.

\begin{theorem}\label{PROP flowing from t=0}
    Let $g(t)$, $\twiddle{g}(t)$, $t\in(0,T]$ be complete Ricci flows defined on manifolds $\mfd$, $\twiddle{\mfd}$. Suppose $g(t)$, $\twiddle{g}(t)$ satisfy the assumptions \eqref{EQN assumptions inj} for some $\alpha>0$, and for all $t\in(0,T]$. Let $d_0,\twiddle{d}_0$ be the uniform limits of $d_{g(t)},d_{\twiddle{g}(t)}$ as $t\searrow0$.
    
    Let $(t_i)_{i\in\nats}$ be a sequence of times in $(0,T]$ with $t_i\searrow0$. For some $\rho>0,0<R<\half$, suppose  we have a map $\ph:\mfd\rightarrow\mfd'$ such that
    \[
    |d_{g(t_i)}(x,y)-d_{\twiddle{g}(t_i)}(\ph(x),\ph(y))|<\eps_i\sqrt{t_i}
    \hspace{5ex} \text{whenever } x,y\in\mfd,\ 0<d_{g(t_i)}(x,y)\leq R\sqrt{t_i\alpha^{-1}},
    \]
    \[
    d_{\twiddle{g}(t_i)}(\ph(x),\ph(y))\geq \rho\sqrt{t_i}
    \hspace{8ex} \text{whenever } x,y\in\mfd, d_{\mfd}(x,y)\geq R\sqrt{t_i\alpha^{-1}},
    \]
    \[
    \text{and }\twiddle{\mfd}=\bigcup_{x\in\mfd}B_{\twiddle{g}(t_i)}(\ph(x),\eps_i\sqrt{t_i}),
    \]
    for each $i$, for some positive sequence $(\eps_i)_{i\in\nats}$. Write $\twiddle{\eps}=\liminf \eps_i$.
     
    Then provided $\twiddle{\eps}$ is sufficiently small, depending only on $n,\alpha,\rho,R$, the following is true. There exists a smooth family of diffeomorphisms $F:(\mfd,g(t))\times(0,\min\{T,1\}]\rightarrow(\twiddle{\mfd},\twiddle{g}(t))$ solving \eqref{EQN HMHF}. Further, $F$ satisfies
    \[
    (1-\eps)g(t)\leq F_{t}^*\twiddle{g}(t) \leq(1+\eps)g(t) \hspace{3ex} \text{on }\mfd,
    \]
    \[
    \twiddle{d}_0(F_{t}(x),\ph(x))\leq C'\sqrt{t}
    \]
    for all $t\in(0,\min\{T,1\}]$, where $\eps=\eps(n,\alpha,\rho,R,\twiddle{\eps})\rightarrow0$ as $\twiddle{\eps}\rightarrow0$, and $C=C'(n,\alpha,\rho,R)>0$. In particular, $F_t\rightarrow \ph$ uniformly as $t\searrow0$.
\end{theorem}

\begin{proof}
    Fix $i\in\nats$ for now. First, we seek to apply Theorem \ref{PROP Katsuda extended} to $\ph$. Let
    \[
    \twiddle{\delta}:=\inf\{\delta:a(n,\alpha,\alpha^{-1/2},\rho,R\alpha^{-1/2};\delta)>\twiddle{\eps}\},
    \]
    where $a$ is as given in Theorem \ref{PROP Katsuda extended}. Note that $\twiddle{\delta}<\infty$ provided $\twiddle{\eps}$ is small enough, depending only on $n,\alpha,R$. Further, $\twiddle{\delta}$ is increasing in $\twiddle{\eps}$, with $\twiddle{\delta}\searrow0$ as $\twiddle{\eps}\searrow0$. By passing to a subsequence, we may assume without loss of generality that $\eps_i<a(n,\alpha,\alpha^{-1/2},\rho,R\alpha^{-1/2};\twiddle{\delta}+2^{-i})$ for each $i$.

    Then we can apply Theorem \ref{PROP Katsuda extended} to $\ph$, considered as a map from $(\mfd,t_i^{-1}g(t_i))$ to $(\twiddle{\mfd},t_i^{-1}\twiddle{g}(t_i))$. This gives a diffeomorphism $f_i:\mfd\rightarrow\twiddle{\mfd}$ which satisfies
    \begin{equation}\label{EQN Getting RDT flow, deriv bd after smoothing, tang vector}
        ||d(f_i)_x(\xi)|_{\twiddle{g}(t_i)}-1|<\twiddle{\delta}+2^{-i} \text{ for every } (x,\xi)\in T\mfd \text{ with }|\xi|_{g(t_i)}=1,
    \end{equation}
    \begin{equation}\label{EQN Getting RDT flow, dist bd after smoothing}
        d_{\twiddle{g}(t_i)}(f_i(x),\ph(x))\leq (\twiddle{\delta}+2^{-i})\sqrt{t_i} \text{ for every } x\in\mfd.
    \end{equation}

    Reducing the upper bound on $\twiddle{\eps}$ if necessary such that $\twiddle{\delta}\leq\frac{1}{4}$, \eqref{EQN Getting RDT flow, deriv bd after smoothing, tang vector} gives
    \begin{equation}\label{EQN Getting RDT flow, deriv bd after smoothing, metric sandwich}
        (1-(\twiddle{\delta}+2^{-i}))g(t_i) \leq f_i^*\twiddle{g}(t_i) \leq (1+(\twiddle{\delta}+2^{-i}))g(t_i).
    \end{equation}

    Now, we apply Proposition \ref{PROP Long global HMHF existence}, taking $t=t_i$ as the start time, with $\Lambda_0=2n$. We note that the finite bounds for the derivatives of $\riem(g(t)),\riem(\twiddle{g}(t))$ when $t\geq t_i>0$ follow from Shi's estimates \cite{Shi}. This gives a smooth solution $F_i:\mfd\times[t_i,\widehat{T}_i]\rightarrow\twiddle{\mfd}$ to \eqref{EQN HMHF}. Writing $F_{i,t}(\cdot):=F_i(\cdot,t)$, we have that $F_{i,t_i}=f_i$, and $0<\widehat{T}_i\leq T$ is the maximal time such that
    \[
    |dF_{i,t}|^2\leq2n
    \]
    for $t\in[0,\widehat{T}_i]$. Further, we have
    \begin{equation}\label{EQN Getting RDT flow, higher deriv bd after flowing}
        |D^kdF_{i,t}|^2\leq L_k(n,\alpha,k)(t-t_i)^{-k},
    \end{equation}
    \begin{equation}\label{EQN Getting RDT flow, dist bd after flowing}
        \twiddle{d}_0(F_{i,t}(x),\ph(x))\leq C(n,\alpha,\rho,R)\sqrt{t}
    \end{equation}
    for all $(x,t)\in\mfd\times[t_i,\widehat{T}_i]$, where for the second inequality we have combined \eqref{EQN Long global HMHF existence, dist result}, \eqref{EQN Getting RDT flow, dist bd after smoothing}, and Lemma \ref{PROP distance distortion}.

    Now, reduce the upper bound on $\twiddle{\eps}$ if necessary so that $\twiddle{\delta}<\frac{1}{16\sqrt{n}}\overline{\eps}_1$, where $\overline{\eps}_1=\overline{\eps}_1(n)>0$ is from Proposition \ref{PROP HMHF short time diffeo}. Then for sufficiently large $i$, we may apply Proposition \ref{PROP HMHF short time diffeo} to obtain $\tau_i=\tau_i(n,\alpha,t_i)>0$ such that $F_{i,t}$ is a diffeomorphism for every $t\in[t_i,t_i+\tau_i]\cap[0,T]$.

    We further reduce the upper bound on $\twiddle{\eps}$, so that $\twiddle{\delta}<\half\overline{\eps}_2$, where $\overline{\eps}_2=\overline{\eps}_2(n,\alpha)>0$ is from Theorem \ref{PROP CLL flow existence}. Then by applying this theorem, we see that for sufficiently large $i$, there is a solution $\widehat{g}_i(t)$ to \eqref{EQN Ricci-DeTurck} on $\twiddle{\mfd}\times[t_i,\min\{T,1\}]$ with $\widehat{g}_i(t_i)=(f_i^{-1})^*g(t_i)$, and such that
    \[
    (1-\Lambda(\twiddle{\delta}+2^{-i}))\twiddle{g}(t) \leq \widehat{g}_i(t) \leq (1+\Lambda(\twiddle{\delta}+2^{-i}))\twiddle{g}(t)
    \]
     for all $t\in[t_i,\min\{T,1\}]$, where $\Lambda=\Lambda(n,\alpha)$. Further, this is the unique solution to \eqref{EQN Ricci-DeTurck} with this initial data, so we see that we must have $\widehat{g}_i(t) \equiv (F_{i,t}^{-1})^*g(t)$ while $F_{i,t}$ continues to be a diffeomorphism. We have already obtained that $F_{i,t}$ is a diffeomorphism for a small positive time, and therefore $(F_{i,t}^{-1})^*g(t)=\widehat{g}_i(t)$ for that time. Reducing the upper bound on $\twiddle{\eps}$ such that $\Lambda\twiddle{\delta}<\frac{1}{4}$, we see that $\Lambda(\twiddle{\delta}+2^{-i})\leq\frac{1}{4}$ for sufficiently large $i$. Then $F_{i,t}$ cannot cease to be a diffeomorphism while $t\leq\min\{T,1\}$, so
    \[
    (1-\Lambda(\twiddle{\delta}+2^{-i}))\twiddle{g}(t) \leq (F_{i,t}^{-1})^*g(t) \leq (1+\Lambda(\twiddle{\delta}+2^{-i}))\twiddle{g}(t),
    \]
    for all $t\in[t_i,\min\{\widehat{T}_i,1\}]$.

    Further, for $t\in[t_i,\min\{\widehat{T}_i,1\}]$, we have
    \[
    |dF_{i,t}|^2\leq n(1+2\Lambda(\twiddle{\delta}+2^{-i}))\leq\frac{3n}{2}.
    \]
    Since $\widehat{T}_i$ is the maximum time such that $|dF_{i,t}|^2\leq2n$, we see that $\widehat{T}_i\geq\min\{T,1\}$. Hence we have that for all $t\in[t_i,\min\{T,1\}]$, the map $F_{i,t}$ satisfies
    \begin{equation}\label{EQN Getting RDT flow, inverse metric sandwich}
        (1-\Lambda(\twiddle{\delta}+2^{-i}))\twiddle{g}(t) \leq (F_{i,t}^{-1})^*g(t) \leq (1+\Lambda(\twiddle{\delta}+2^{-i}))\twiddle{g}(t),
    \end{equation}
    and therefore
    \begin{equation}\label{EQN Getting RDT flow, bound on deriv of inverse}
        |dF_{i,t}^{-1}|^2\leq n(1+\Lambda(\twiddle{\delta}+2^{-i}))\leq\frac{5n}{4}
    \end{equation}
    and
    \begin{equation}\label{EQN Getting RDT flow, metric sandwich after RDT comparison}
        (1-2\Lambda(\twiddle{\delta}+2^{-i}))g(t) \leq F_{i,t}^*\twiddle{g}(t) \leq (1+2\Lambda(\twiddle{\delta}+2^{-i}))g(t).
    \end{equation}
    From \eqref{EQN Getting RDT flow, inverse metric sandwich} and \eqref{EQN Getting RDT flow, metric sandwich after RDT comparison}, we also obtain that if $x,y\in\mfd$ and $t\in[t_i,\min\{T,1\}]$, then
    \begin{equation}\label{EQN Getting RDT flow, dist distortion}
        (1-2\Lambda(\twiddle{\delta}+2^{-i}))d_{g(t)}(x,y) \leq d_{\twiddle{g}(t)}(F_{i,t}(x),F_{i,t}(y)) \leq (1+2\Lambda(\twiddle{\delta}+2^{-i}))d_{g(t)}(x,y).
    \end{equation}

    Finally, we take a subsequential limit of the $F_i$ as $i\rightarrow\infty$ to obtain the required $F$, with \eqref{EQN Getting RDT flow, higher deriv bd after flowing} and \eqref{EQN HMHF} allowing us to take a smooth limit for $t>0$. Writing $\eps:=2\Lambda\twiddle{\delta}$, the required estimates for $F$ follow from \eqref{EQN Getting RDT flow, metric sandwich after RDT comparison} and \eqref{EQN Getting RDT flow, dist bd after flowing}, in particular showing that $F_t$ is a local diffeomorphism for all $t\in(0,\min\{T,1\}]$. Injectivity of $F_t$ for $t\in(0,\min\{T,1\}]$ follows from \eqref{EQN Getting RDT flow, dist distortion}, while \eqref{EQN Getting RDT flow, bound on deriv of inverse} allows us to take the limit of $F_{i,t}^{-1}$ to obtain $F_t^{-1}$, giving surjectivity.
\end{proof}

\section{Obtaining uniqueness}\label{SEC uniqueness}

In this final section, we will combine the uniform Reifenberg condition with the results of the previous section to prove the main result. First, note that it is possible to convert the Reifenberg assumption into control on $|\riem|$ which is slightly stronger than assumed in \eqref{EQN assumptions inj}.

\begin{lemma}[Deruelle--Schulze--Simon, {\cite[Lemma A.2]{DSS2}}]\label{PROP Reifenberg curvature bound}
  Let $(M^n,g(t))_{t\in (0,T]}$, $T\leq 1$, be a smooth Ricci flow, satisfying $\ric(g(t)) \geq -g(t) , \quad |\riem(g(t))| \leq \alpha/t$, where 
 $M$ is connected but $(M^n,g(t))$ not  necessarily complete. Assume furthermore that $B_{g(t)}(x_0,1)$ is compactly contained in $M$ for all $t\in (0,T],$ and let $X\coloneqq  (\cap_{s \in (0,T]} B_{g(s)}(x_0,\frac 1 2))$ be as in Lemma \ref{PROP distance distortion} and endowed with the metric $d_0$.
 Assume that $ B_{d_0}(x_0,r)$ is as in Lemma \ref{PROP distance distortion}, and furthermore that all tangent cones of $ B_{d_0}(x_0,r)$ are isometric to  $(\reals^n,d(\delta)).$ 
Let $\eps>0$ be given. Then there exists a $\widehat{\delta}>0$ depending on $\eps,n,x_0$ and   the solution, such that if 
$0<t\leq \widehat{\delta}$ then 
\begin{equation*}
|\riem(g(t))| \leq  \frac{\eps}{t},\quad\text{ on  $ B_{d_0}\left(x_0,\tfrac r 4\right)$.}
\end{equation*} 
 \end{lemma}

By examining the proof of \cite[Lemma A.2]{DSS2}, we see that if $(\mfd,d_0)$ is uniformly Reifenberg, then $\widehat{\delta}$ no longer depends on $x_0$. Hence when $(\mfd,d_0)$ is uniformly Reifenberg, we may assume that both $g(t)$ and $\twiddle{g}(t)$ satisfy
\[
|\riem|\leq\frac{\alpha(t)}{t},
\]
where $\alpha(t)\rightarrow0$ as $t\searrow0$.

We are now ready to prove Theorem \ref{PROP Reifenberg uniqueness}.

\begin{proof}[Proof of Theorem \ref{PROP Reifenberg uniqueness}.]
    As above, the uniform Reifenberg condition gives that $g(t)$, $\twiddle{g}(t)$ both satisfy
    \[
    |\riem|\leq\frac{\alpha(t)}{t},
    \]
    where $\alpha\rightarrow0$ as $t\searrow0$.

    Consider some fixed $t\in(0,\min\{T,1\}]$. By Lemma \ref{PROP distance distortion}, we have that
    \[
    d_0 - \gamma(n)\sqrt{\alpha(t)t} \leq d_{g(t)} \leq e^{(n-1)t}d_0,
    \]
    \[
    \twiddle{d}_0 - \gamma(n)\sqrt{\alpha(t)t} \leq d_{\twiddle{g}(t)} \leq e^{(n-1)t}\twiddle{d}_0,
    \]
    so
    \[
    | d_{g(t)} - \ph^*d_{\twiddle{g}(t)} |
    \leq C_n \left( t\,d_{g(t)} + \sqrt{\alpha(t)t} \right)
    \]
    since $d_0=\ph^*\twiddle{d}_0$ and $t\leq1$.
    Therefore if $d_{g(t)}\leq\half\sqrt{\alpha_0^{-1}t}$, where $\alpha_0:=\sup_{t\in(0,\min\{T,1\}]}\alpha(t)$, then
    \[
    | d_{g(t)} - \ph^*d_{\twiddle{g}(t)} |
    \leq C_n \left( \frac{t}{\sqrt{\alpha_0}} + \sqrt{\alpha(t)} \right) \sqrt{t}.
    \]
    
    Since $\alpha_0^{-\half}t + \sqrt{\alpha(t)}\rightarrow0$ as $t\searrow0$, and $\ph$ is surjective, we may apply Theorem \ref{PROP flowing from t=0} with $\alpha=\alpha_0$ to obtain a solution $F:(\mfd,g(t))\times(0,\min\{T,1\}]\rightarrow(\twiddle{\mfd},\twiddle{g}(t))$ to \eqref{EQN HMHF}, such that
    \[
    g(t) = F_t^*\,\twiddle{g}(t)
    \]
    for all $t\in(0,\min\{T,1\}]$, and $F_t(x)\rightarrow \ph(x)$ as $t\searrow0$ for all $x\in\mfd$.

    Since $F_t:(\mfd,g(t))\rightarrow(\twiddle{\mfd},\twiddle{g}(t))$ is an isometry of Riemannian manifolds for each $t\in(0,\min\{T,1\}]$, we have $\lap_{g(t),\twiddle{g}(t)} F_t=0$. Then $F$ is constant in time by \eqref{EQN HMHF}, so $F_t\equiv\ph$.
\end{proof}

\appendix
\section{Appendix}

We use this appendix to collect various results which are of use in \S\ref{SEC HMHF} and \S\ref{SEC R-DT flow}. First, we have several results from \cite{MCLeePaper}, concerning solutions to the harmonic map heat flow.

\begin{proposition}[Lee, {\cite[Proposition 2.2]{MCLeePaper}}]\label{PROP Local HMHF existence}
    Suppose $g(t), t\in[0,T]$ is a smooth Ricci flow on some open set $U\Subset\mfd^n$ (not necessarily complete), and $\twiddle{g}(t), t\in[0,T]$ is a smooth complete Ricci flow on $\twiddle{\mfd}^n$, such that for $t\in[0,T]$,
    \begin{enumerate}[label=(\roman*)]
        \item $\sup_U|\riem(g(t))| + \sup_{\twiddle{\mfd}}|\riem(\twiddle{g}(t))| \leq \rho^{-2}$,
        \item $\sup_U|\nabla\riem(g(t))| \leq \rho^{-3}$,
        \item $\sup_{\twiddle{\mfd}}|\twiddle{\nabla}^k\riem(\twiddle{g}(t))| < \infty$ for all $k\geq0$,
    \end{enumerate}
    for some $\rho>0$. Let $\Lambda_0>0$, and let $f:\overline{U}\rightarrow\twiddle{\mfd}$ be a smooth map with $|df|_{g(0),\twiddle{g}(0)}^2\leq\Lambda_0$. Write $U_\rho:=\{x\in U:B_{g(0)}(x,\rho)\Subset U\}$. Then there exists $T_0(n,\Lambda_0)>0$, and a smooth solution $F:U_\rho\times[0,\min\{T,T_0\rho^2\}]\rightarrow\mfd$ to \eqref{EQN HMHF}, such that $F_0=f$ on $U$ and
    \[
    |dF_t|^2\leq10\Lambda_0
    \]
    on $U_\rho\times[0,\min\{T,T_0\rho^2\}]$.
\end{proposition}

\begin{proposition}[Lee, {\cite[Proposition 2.1]{MCLeePaper}}]\label{PROP Lee HMHF bound on derivs}
    Suppose $(\mfd^n, g(t))$, $(\twiddle{\mfd}^m, \twiddle{g}(t))$, $t \in [0, T]$ are two smooth complete solutions to the Ricci flow, such that $|\riem(g(t))| + |\riem(\twiddle{g}(t))| \leq \alpha t^{-1}$ for $t\in(0,T]$, for some $\alpha>0$. Let $x_0\in\mfd$, $r>0$ be such that $B_{g(0)}(x_0,4r)\Subset\mfd$.
    
    Suppose $F:B_{g(0)}(x_0,4r)\times[0,T] \rightarrow \twiddle{\mfd}$ is a smooth solution to \eqref{EQN HMHF} such that $|dF_t|^2\leq\Lambda_0$ for all $t\in[0,T]$ for some $\Lambda_0>0$. Then for all $k\in\nats$, there exists $L_k(n,m,\alpha,\Lambda_0,k),S_k(n,m,\alpha,\Lambda_0,k)>0$ such that
    \[
    |D^kdF_t|^2 \leq L_kt^{-k}
    \]
    for all $(x,t)\in B_{g(0)}(x_0,r)\times(0,\min\{T,S_kr^2\}]$.
\end{proposition}

\begin{lemma}[Lee, {\cite[Lemma 2.2, second part]{MCLeePaper}}]\label{PROP Lee HMHF evo eqn for derivs}
    Suppose $(\mfd^n, g(t))$, $(\twiddle{\mfd}^m, \twiddle{g}(t))$, $t \in [0, T]$ are two smooth solutions to the Ricci flow.
    
    If $F:B_{g(0)}(x_0,4r)\times[0,T] \rightarrow \twiddle{\mfd}$ is a smooth solution to \eqref{EQN HMHF}, then for all $k\in\nats$, we have
    \begin{align*}
        \left(\del_t-\lap_{g(t)}\right)|D^kdF_t|^2 \leq 
        &-2|D^{k+1}dF_t|^2 + C(n,m,k)|D^kdF_t|\Bigg[\sum_{\ell=0}^k|\nabla^\ell\riem(g(t))||D^{k-\ell}dF_t| \\
        &+ \sum_{\ell=0}^k\sum_{p=0}^\ell\sum_{i+j=p}|\twiddle{\nabla}^{\ell-p}\riem(\twiddle{g}(t))||D^idF_t||D^jdF_t||D^{k-\ell}dF_t|\\
        &+\sum_{\ell=0}^k\sum_{p=0}^\ell|\twiddle{\nabla}^{\ell-p}\riem(\twiddle{g}(t))||D^pdF_t||D^{k-\ell}dF_t| + |\riem(\twiddle{g}(t))||D^kdF_t| \Bigg].
    \end{align*}
\end{lemma}

Next, we will need some variations of \cite[Corollary 3.1]{LeeTamMaxPrinc}, a localised maximum principle for Ricci flows. Note that when differential inequalities are said to hold \emph{in the barrier sense}, we will be using the same definition as \cite[Definition 18.5]{RFT&A}. In particular, given a (possibly nonsmooth) function $f$ and $C\in\reals$, we say that $\heatop f\leq C$ at $(x_0,t_0)$ if the following is true. For any $\eps>0$ there is a $C^2$ function $F$ with $F\leq f$ defined in a neighbourhood of $(x_0,t_0)$, and such that $F(x_0,t_0)=f(x_0,t_0)$, and $\heatop F\leq C+\eps$ at $(x_0,t_0)$.

First, we obtain the following variation of \cite[Corollary 3.1]{LeeTamMaxPrinc}, to allow for the case where we have a slightly weaker initial bound on $\ph$. This can be obtained by a minor adaptation of the proof, as outlined below. This will be improved in upcoming work \cite{SharpMaxPrinc}.

\begin{theorem}\label{PROP Lee Tam max princ improved}
    Let $(\mfd^n,g(t)),t\in[0, T]$ be a smooth solution to the Ricci flow which is possibly incomplete. Suppose
    \[
    \ric(g(t))\leq \alpha t^{-1}
    \]
    on $\mfd\times(0,T]$ for some $\alpha>0$. Let $\ph(x,t)$ be a continuous function on $\mfd\times[0,T]$ which satisfies $\ph(x,t)\leq\alpha t^{-b}$ for some $b\geq0$ on $\mfd\times(0, T]$, and
    \[
    \left(\del_t-\lap_{g(t)}\right)\ph\,\Big|_{(x_0,t_0)} \leq L(x_0,t_0)\ph(x_0,t_0)
    \]
    whenever $\ph(x_0,t_0)>0$ in the sense of barrier, for some continuous function $L(x, t)$ on $\mfd\times[0, T]$ with $L(x, t) \leq\alpha t^{-1}$. Let $p\in\mfd, r>0$ be such that $B_{g(0)}(p, r)\Subset\mfd$ and $\ph(x, 0)\leq0$ on $B_{g(0)}(p, r)$. Then for any $\ell>\alpha+b$, there exists $\widehat{T}=\widehat{T}(n,\alpha,\ell,b)>0$ such that for $t\in[0,\min\{T,\widehat{T}r^2\}]$,
    \[
    \ph(p,t) \leq r^{-2(\ell+b)}t^\ell.
    \]
\end{theorem}

\begin{proof}
    We first consider the case $r=2$, and proceed with an almost identical proof to that of \cite[Theorem 1.1]{LeeTamMaxPrinc}. The proof is entirely unchanged until the point where we take $\eta=\half t^\ell$ (near the end of the proof; see \cite[\S3]{LeeTamMaxPrinc} for more details). Then with the bound $\ph\leq\alpha t^{-b}$, the subsequent line becomes
    \[
    \ell t_0^{\ell-1}\leq t_0^\ell\bigg(\frac{\alpha}{t_0}+C_1^2m^2\left(\frac{\alpha}{t_0^{\ell+b}}\right)^\frac{2}{m}\bigg),
    \]
    which implies
    \[
    t_0^{\ell-1}\leq C_1^2m^2\alpha^\frac{2}{m}t_0^{\ell-\frac{2}{m}(l+b)}.
    \]
    Now choosing $m$ large enough that $\frac{2}{m}(\ell+b)<\half$ gives that $t_0\geq t_2(n,\alpha,\ell,b)$. Concluding as in the proof of \cite[Theorem 1.1]{LeeTamMaxPrinc} gives the result in the case $r=2$.

    To obtain the result for all $r$, we rescale as in the proof of \cite[Corollary 3.1]{LeeTamMaxPrinc}. Taking $\lambda=4r^{-2}$, we define $g_1(x,t)=\lambda g(x,\lambda^{-1}t)$, $\ph_1(x,t)=\lambda^{-b}\ph(x,\lambda^{-1}t)$, $L_1(x,t)=\lambda^{-1}L(x,\lambda^{-1}t)$. Then $g_1, \ph_1, L_1$ satisfy the conditions needed to apply the result in the case $r=2$.
\end{proof}

Next, we will need another version of \cite[Corollary 3.1]{LeeTamMaxPrinc} which works for a ``twisted" heat operator. To this end, we first establish the following adaptation of \cite[Lemma 8.3a]{PerelmanDistanceEvolution}.

\begin{lemma}\label{PROP twisted dist evo}
    Let $(\mfd,g(t))_{t\in[0,T]}$ be a Ricci flow. Suppose that at some time $t_0$, we have that
    $$
    |\riem(g(t_0))| \leq K \hspace{3ex}\text{in}\hspace{3ex} B_{g(t_0)}(p,r),
    $$
    for some number $K>0$, and some point $p$ and radius $r>0$ such that $B_{g(t_0)}(p,r)\Subset\mfd$. Let $0<r_0\leq\half r$.
    Finally, let $\widehat{g}$ be a Riemannian metric defined on $B_{g(t_0)}(p,r)$, which satisfies $(1-\eps)g(t_0)\leq\widehat{g}\leq(1+\eps)g(t_0)$ for some $0<\eps<1$. Then there is some $C_\eps>0$, depending only on $\eps$, such that $d(x,t):=\operatorname{dist}_{g(t)}(p,x)$ satisfies
    \[
    \left(\del_t-\widehat{g}^{ij}\nabla_i\nabla_j\right) d 
    \geq -C_\eps(n-1)\left(Kr_0 + r_0^{-1}\right)
    \hspace{3ex}\text{in } B_{g(t_0)}(p,r)\setminus B_{g(t_0)}(p,2r_0)
    \]
    in the barrier sense, where $\nabla$ is defined with respect to $g(t_0)$.
\end{lemma}

\begin{proof}
    We follow an almost identical proof to that of \cite[Lemma 8.3a]{PerelmanDistanceEvolution}, but we give the full details for convenience. First, let $t_0\in[0,T]$ and $x_0\in B_{g(t_0)}(p,r)\setminus B_{g(t_0)}(p,2r_0)$, and let $\gamma$ be the shortest geodesic from $p$ to $x_0$, parametrised by arc length. Next, at the point $x_0$, take an orthonormal basis $\{e_i\}_{i=1}^n$ with respect to $g(t_0)$ such that $\widehat{g}$ is diagonal, that is $\widehat{g}_{ij}=\lambda_i\delta_{ij}$ for some $\lambda_1,\ldots,\lambda_n\in(0,2)$. Then move this along $\gamma$ by parallel transport with respect to $g(t_0)$, to obtain an orthonormal basis $\{e_i(s)\}_{i=1}^n$ along $\gamma$.

    For fixed $i$, define the vector field $V_i(s)$ along $\gamma$ by $V_i(s)=f_i(s)e_i(s)$, where
    \[
    f_i(s)=
    \begin{cases}
        \frac{s}{r_0} & 0\leq s<r_0 \\
        1 & r_0\leq s<d_0-r_0 \\
        1+(\lambda_i^{-1/2}-1)\frac{s+r_0-d_0}{r_0} & d_0-r_0\leq s\leq d_0,
    \end{cases}
    \]
    writing $d_0:=d_{g(t_0)}(x_0,t_0)$. This variation field gives rise to an upper barrier for $d$ at $(x_0,t_0)$, so its second variation $s_{V_i}''(\gamma)$ satisfies
    \[
    \lambda_i^{-1}\nabla_i\nabla_id\leq s_{V_i}''(\gamma).
    \]
    The barrier is constructed in such a way that the endpoint contributions to the second variation formula are $0$. Therefore, writing $e_i=e_i(s)$, $g=g(t_0)$, $\riem=\riem(g(t_0))$, $\mu_i:=\lambda_i^{-1/2}$ for simplicity, and writing $\widehat{e}_i:=e_i-g(e_i,\gamma')\gamma'$ for the part of $e_i$ perpendicular to $\gamma'$, the second variation formula gives
    \begin{align*}
        \lambda_i^{-1}\nabla_i&\nabla_id \\
        &\leq \int_0^{r_0} \left( \frac{1}{r_0^2}g(\widehat{e}_i,\widehat{e}_i) - \frac{s^2}{r_0^2}\riem(\widehat{e}_i,\gamma',\widehat{e}_i,\gamma') \right) \,\mathrm{d}s 
        + \int_{r_0}^{d_0-r_0} - \riem(\widehat{e}_i,\gamma',\widehat{e}_i,\gamma') \,\mathrm{d}s \\
        &\hspace{3ex}+ \int_{d_0-r_0}^{d_0} \left( \frac{(\mu_i-1)^2}{r_0^2}g(\widehat{e}_i,\widehat{e}_i) - \bigg(1+(\mu_i-1)\frac{(s+r_0-d_0)}{r_0}\right)^2\riem(\widehat{e}_i,\gamma',\widehat{e}_i,\gamma') \bigg) \,\mathrm{d}s  \\
        &= - \int_0^{d_0} \riem(e_i,\gamma',e_i,\gamma') \,\mathrm{d}s
        + \frac{1}{r_0}\left(1+(\mu_i-1)^2\right)\left(1-g(e_i,\gamma')^2\right) \\
        &\hspace{3ex}+ \int_0^{r_0} \left( 1 - \frac{s^2}{r_0^2}\right)\riem(\widehat{e}_i,\gamma',\widehat{e}_i,\gamma')  \,\mathrm{d}s \\
        &\hspace{3ex}+ \int_{d_0-r_0}^{d_0} -\left( 2(\mu_i-1)\frac{(s+r_0-d_0)}{r_0} + (\mu_i-1)^2\frac{(s+r_0-d_0)^2}{r_0^2}\right)\riem(\widehat{e}_i,\gamma',\widehat{e}_i,\gamma') \,\mathrm{d}s \\
        &\leq - \int_0^{d_0} \riem(e_i,\gamma',e_i,\gamma') \,\mathrm{d}s
        + C_\eps\left(r_0^{-1}+Kr_0\right)\left(1-g(e_i,\gamma')^2\right).
    \end{align*}
    
    Now summing over $i$, and noting that $\del_td=-\int_\gamma\ric(\gamma',\gamma')$, gives the result.
\end{proof}

Using this, we may generalise Theorem \ref{PROP Lee Tam max princ improved} to the case of this operator.

\begin{theorem}\label{PROP Lee Tam max princ improved twisted}
    Let $(\mfd^n,g(t)),t\in[0, T]$ be a smooth solution to the Ricci flow which is possibly incomplete. Suppose
    \[
    |\riem(g(t))|\leq \alpha t^{-1}
    \]
    on $\mfd\times(0,T]$ for some $\alpha>0$, and let $(\widehat{g}(t))_{t\in[0,T]}$ be a family of Riemannian metrics defined on $B_{g(t_0)}(p,r)$, which satisfies $(1-\eps)g(t) \leq\widehat{g}(t) \leq(1+\eps)g(t)$ for all $t$, for some $0<\eps<1$. Let $\ph(x,t)$ be a continuous function on $\mfd\times[0,T]$ which satisfies $\ph(x,t)\leq\alpha t^{-b}$ for some $b\geq0$ on $\mfd\times(0, T]$, and
    \[
    \left(\del_t-\widehat{g}^{ij}\nabla_i\nabla_j\right)\ph\,\Big|_{(x_0,t_0)} \leq L(x_0,t_0)\ph(x_0,t_0)
    \]
    whenever $\ph(x_0,t_0)>0$ in the sense of barrier, for some continuous function $L(x, t)$ on $\mfd\times[0, T]$ with $L(x, t) \leq\alpha t^{-1}$. Here $\nabla$ is always defined with respect to $g(t)$. Let $p\in\mfd, r>0$ be such that $B_{g(0)}(p, r)\Subset\mfd$ and $\ph(x, 0)\leq0$ on $B_{g(0)}(p, r)$. Then for any $\ell>\alpha+b$, there exists $\widehat{T}(n,\eps,\alpha,\ell,b)>0$ such that for $t\in[0,\min\{T,\widehat{T}r^2\}]$,
    \[
    \ph(p,t) \leq r^{-2(\ell+b)}t^\ell.
    \]
\end{theorem}

\begin{proof}
    The proof is again almost identical to that of \cite[Theorem 1.1]{LeeTamMaxPrinc}. We begin by replacing the use of \cite[Lemma 8.3a]{PerelmanDistanceEvolution} with Lemma \ref{PROP twisted dist evo}, to see that in our case,
    \[
    \left(\del_t-\widehat{g}^{ij}\nabla_i\nabla_j\right) d 
    \geq -C_\eps(n-1)\sqrt{\frac{\alpha}{t}}
    \]
    outside $B_{g(t)}(p,2\sqrt{\alpha^{-1}t})$, for some $C_\eps>0$ depending only on $\eps$. Hence we may use this to define a radial cutoff function $\Phi_r$ with the property
    \[
    \left(\del_t-\widehat{g}^{ij}\nabla_i\nabla_j\right) \Phi_r \leq0,
    \]
    and then calculate
    \[
    \left(\del_t-\widehat{g}^{ij}\nabla_i\nabla_j\right) (-\Phi_r\phi+\eta)
    \]
    for some candidate bound $\eta$, exactly as in \cite{LeeTamMaxPrinc} (up to changing some constants, which now depend on $\eps$).
    
    The rest of the proof of \cite[Theorem 1.1]{LeeTamMaxPrinc} may continue unchanged, although once more we make the modifications outlined in the proof of Theorem \ref{PROP Lee Tam max princ improved}.
\end{proof}

\printbibliography

\end{document}